\documentclass[11pt]{article}%

\usepackage{amsmath,amssymb,amsfonts}%
\usepackage{a4wide}%
\usepackage{theorem}%
\usepackage{epsfig,subfigure}%
\usepackage{color}%
\usepackage{hyperref}%
\usepackage{enumitem}
\usepackage{authblk}

\theoremstyle{change}%

\newtheorem{definition}{Definition}[section]%
\newtheorem{theorem}[definition]{Theorem}%
\newtheorem{lemma}[definition]{Lemma}%

\newtheorem{VP}[definition]{Verification Problem}%
{\theorembodyfont{\rmfamily} \newtheorem{remark}[definition]{Remark}}%

\newenvironment{proof}{{\bf Proof:}}{\qquad \hspace*{\fill} $\Box$}%

\newcommand{\bx}{x}

\newcommand{\fS}{\mathfrak S}

\newcommand{\balpha}{{\boldsymbol \alpha}}

\newcommand{\dist}{\operatorname{dist}}

\newcommand{\co}{\operatorname{co}}

\newcommand{\cA}{{\mathcal A}}

\newcommand{\cD}{{\mathcal D}}

\newcommand{\cH}{{\mathcal H}}

\newcommand{\cL}{{\mathcal L}}

\newcommand{\cR}{{\mathcal R}}

\newcommand{\cT}{{\mathcal T}}

\newcommand{\cV}{{\mathcal V}}

\newcommand{\T}{{\mathfrak S}}

\newcommand{\diam}{\mathrm{diam}}

\newcommand{\R}{\mathbb{R}}%
\newcommand{\Z}{\mathbb{Z}}%
\newcommand{\Sb}{\mathbb{S}}
\newcommand{\N}{\mathbb{N}}%
\newcommand{\CPA}{\operatorname{CPA}}

\newcommand{\norm}[2]{\left \| #1 \right \|_{#2}}

\numberwithin{equation}{section}

\begin{document}
\title{Computation and verification of contraction metrics for \\
exponentially stable equilibria}

\author[1]{Peter Giesl\thanks{email: \texttt{p.a.giesl@sussex.ac.uk}}}
\author[2]{Sigurdur Hafstein\thanks{email: \texttt{shafstein@hi.is}}}
\author[2]{Iman Mehrabinezhad\thanks{email: \texttt{imehrabinzhad@hi.is}}}
\affil[1]{{\small Department of Mathematics, University of Sussex, Falmer BN1 9QH, United Kingdom}}
\affil[2]{{\small Faculty of Physical Sciences, University of Iceland, Dunhagi 5, IS-107 Reykjavik, Iceland}}

%

\date{\today}

\maketitle

\begin{abstract}
The determination of exponentially stable equilibria and their basin of attraction for a dynamical system given by a general autonomous ordinary differential equation can be achieved by means of a contraction metric. A contraction metric is a Riemannian metric with respect to which the distance between adjacent solutions decreases as time increases. The Riemannian metric can be expressed by a matrix-valued function on the phase space.

The determination of a contraction metric can be achieved by approximately solving a matrix-valued partial differential equation by mesh-free collocation using Radial Basis Functions (RBF). However, so far no rigorous verification that the computed metric is indeed a contraction metric has been provided.

In this paper, we combine the RBF method to compute a contraction metric with the CPA method to rigorously verify it. In particular, the computed contraction metric is interpolated by a continuous piecewise affine (CPA) metric at the vertices of a fixed triangulation, and by checking finitely many inequalities, we can verify that the interpolation is a contraction metric. Moreover, we show that, using sufficiently dense collocation points and a sufficiently fine triangulation, we always succeed with the construction and verification. We apply the method to two examples.
\end{abstract}

\vspace*{1ex}



\section{Introduction}
Consider an ordinary differential equation (ODE) of the form%
\begin{equation}\label{eq_ode}
  \dot{x} = f(x),\quad x \in \R^n
\end{equation}
with a $C^s$-vector field $f:\R^n \rightarrow \R^n$ with $s\ge 3$ and further assumptions will be made later.
The solution $x(t)$ with initial value $x(0) = \xi $ is denoted by $S_t \xi:=x(t)$ and is assumed to exist for all $t \geq 0$.
An equilibrium of the ODE is a point $ x_0 \in \R^n $ such that $ f(x_0) = 0 $, from which $x(t)=S_tx_0=x_0$ for all $t\in\R$ follows.
The equilibrium is said to be 
\textit{exponentially stable} if 
there exist $ \alpha,\beta,\delta >0 $ such that
$ \|x(0) - x_{0}\|_2< \delta $ implies
$$ \| x(t)-x_0\|_2 \leq \alpha  \| x(0)-x_0\|_2 e^{-\beta t}\ \ \ \text{for all $ t \geq 0$}. $$
We denote by $ \displaystyle \cA(x_0) = \{ x \in \R^n \,:\,\lim_{t \to \infty} S_t x = x_0 \} $ its \textit{basin of attraction}. \\
For a given domain in $\R^n$, we are interested in proving the existence, uniqueness and exponential stability of an equilibrium, as well as to determine or estimate  its basin of attraction.
There are different methods in the literature towards this problem.  If the position of an equilibrium is known, its exponential stability and a lower bound on its basin of attraction can be obtained
by computing a Lyapunov function for the system \cite{hahn,Khalil2002,lyapunovbook1907endtrans,vidya2002}.  Computing a Lyapunov function analytically is usually not feasible for a nonlinear system, therefore a plethora of numerical methods has been developed.  To name a few, a sum of squared (SOS) polynomials Lyapunov function can be parameterized by using semidefinite programming \cite{AnPaXXX,chesibook,Papachristodoulou2013,sosphdthesis2000parrilo} or with different methods \cite{KaPeXXXXalttosos,RaSh2010branchrelaxlya,Vannelli1985}, an approximate solution to the Zubov equation \cite{zubovbook1964} can be obtained using series expansion \cite{Vannelli1985, zubovbook1964} or by using radial basis functions (RBF) \cite{rbf2007giesl}, or linear programming can be used to parameterize a
continuous and piecewise affine (CPA) Lyapunov function \cite{Giesl2014,Haf2004exex,Johansen2000collo,JGD1999cpajulian,Mar2002cpa} or to verify a Lyapunov function candidate computed by other methods \cite{BGHKLXXXXmultiattr,GiHa2015combi,HaVa2019intLya}.

Another approach is based on contraction metrics, which has the advantage that the position of the asymptotically stable equilibrium, or more generally the attractor, is not needed.  Whereas a Lyapunov function demonstrates that
solutions get closer to the attractor as the time evolves if measured by an appropriate metric on $\R^n$, a contraction metric proves that adjacent solutions come closer  to each other as the time evolves when measured by an appropriate Riemannian metric \cite{Aghannan2003,hahn,krasovskii,lohmiller}.

The analytical computation of a contraction metric is even more difficult than that of a Lyapunov function.  Much less literature is available on numerical methods to compute contraction metrics, but the methods are
often similar to the ones used to compute Lyapunov functions, e.g.~in \cite{Aylward2008} SOS polynomials are parameterized, in \cite{giesl2018kernel,giesl-cont-1} RBF is used, and in \cite{GiHa2013CPAmetric} a CPA contraction
metric is parameterized using semidefinite programming. While the RBF method is computationally more efficient than the CPA method, the RBF method does not guarantee a\,priori  that the computed metric is indeed a contraction metric. However, error estimates show that this holds true if the set of collocation points is sufficiently dense. The CPA method, on the other hand, provides such a guarantee. The goal of this paper is to combine these two methods and their advantages.

In \cite{GiHa2015combi}  the RBF method is used to compute a Lyapunov function candidate with a subsequent verification of the linear constraints of a feasibility problem
for a CPA Lyapunov function. In this paper we follow a similar approach, but for contraction metrics and not for Lyapunov functions.  The result is a method that combines the computational efficiency of the RBF method and the rigour of the CPA method, but for contraction metrics the computational advantage is even larger than for Lyapunov functions,
 because the feasibility problem is not a linear programming problem but a semi-definite optimization problem.

In more detail: We first compute an approximation to a contraction metric using collocation with RBFs and then compute a CPA interpolation of this approximation. Using a sufficiently  dense set of collocation points and a sufficiently  fine triangulation, we are able to prove the conditions for a contraction metric for the CPA interpolation.

The paper is organized as follows: in Section \ref{preliminaries} some basic definitions and results about Riemannian contraction metrics are provided, in particular Theorem \ref{th:1} that describes their implications and  Theorem \ref{un} that shows the existence of a contraction metric as the solution to a matrix-valued PDE.\\

In Section \ref{RBF-section}, we  introduce the optimal recovery problem and review  previous results on finding the unique solution to this problem and the error estimates in recovering the contraction metric. \\
In Section \ref{CPA-section}, we review the concepts needed for the CPA interpolation of a function and then combine all the previous results to obtain Theorem \ref{le1},  which shows that if the constraints of Verification Problem \ref{VP} are satisfied for certain values, then these define a Riemannian contraction metric, and Theorem \ref{RBF-CPA contraction metric}, showing that such values can be obtained from a solution to the optimal recovery problem, if both the triangulation and the collocation points are sufficiently fine.
 \\
In Section \ref{examples-section}, we first present our  algorithm and then provide two examples to demonstrate its applicability. These examples have been used in different references and thus provide a good comparison.

\section{Preliminaries} \label{preliminaries}
In this section we will review basic concepts about Riemannian contraction metrics and some important tools that we will use later on in this paper. As the Riemannian metric which we later calculate will not be differentiable, we give a definition which does not require that.

\begin{definition} [Riemannian metric] \label{metric}  \hspace*{3cm} \\
Let $ G $ be an open subset of $\R^n$. A Riemannian metric is a locally Lipschitz continuous matrix-valued function
$ M:G \to \Sb^{n\times n}$, where $ \Sb^{n\times n} $ denotes the symmetric $ n \times n $ matrices with real entries, such that $ M(x) $ is positive definite for all $ x \in G $. \\
Then, $ \langle v, w\rangle_{M(x)} := v^T M(x) w $ defines a (point-dependent) scalar product for each $ x \in G $ and $ v,w \in \R^n $.

The forward orbital derivative $M'_{+} (x) $ with respect to \eqref{eq_ode} at $ x \in G $ is defined by
\begin{equation}\label{fod}
M'_{+}(x) := \limsup_{h \to 0^+} \dfrac{M \big(S_h x \big) - M(x)}{h}.
\end{equation}
\end{definition}

\begin{remark} \label{OD/FOD}
Note that the forward orbital derivative \eqref{fod} is formulated using a Dini derivative similar to \cite[Definition 3.1]{GiHa2013CPAmetric} and always exists in $ \R \cup \{ \infty \}$. This assumption is less restrictive than \cite[Definition 2.1]{contr2015giesl}, which is the existence and continuity of
$$ M'(x) = \left.\dfrac{d}{dt} M (S_t x ) \right|_{t = 0} . $$
A sufficient condition for the existence and continuity of $M'(x)$  is that
$ M\in C^1(G;\Sb^{n\times n}) $; then
$(M'_{+}(x))_{ij} = (M'(x))_{ij} = \left(\nabla M_{ij}(x)\cdot f(x)\right)_{ij}$ for all $ i, j \in \{ 1, 2, . . . , n \} $. \\
It is also worth mentioning that if $ K \subset G$ is compact, then $M$ in Definition \ref{metric} is uniformly positive definite on $K$, i.e.~there exists an $ \epsilon > 0 $ such that $ v^T M(x) v \geq \epsilon \|v\|^2 $ for all $ v \in \R^n $ and
all $ x \in K $.
\end{remark}

\begin{remark}\label{for-orb-der-formula}
It is useful to have a more accessible expression for the forward orbital derivative in terms of  $ f $ in \eqref{eq_ode}. In fact
we have
\begin{eqnarray*}
 M_{+}'(x) := \limsup_{h \to 0^+} \dfrac{M\big(S_h x \big) - M(x)}{h}
= \limsup_{h \to 0^+} \dfrac{M \big(x + h f(x) \big) - M(x)}{h},
\end{eqnarray*}
because by  \cite[Lemma 3.3]{GiHa2013CPAmetric} an analogous formula holds true for each entry $M_{ij}$ of the matrix $M$.
\end{remark}

\begin{definition} [Riemannian contraction metric] \label{contraction metric} \cite{contr2015giesl}\\
Let $K $ be a compact subset of an open set $ G \subset \R^n  $ and
$ M\in C^0(G;\Sb^{n\times n})$ be a Riemannian metric. For $x\in K, v \in \R^n$ define
$$ L_M(x; v) := \dfrac{1}{2} v^T \big [ M(x) Df(x) + Df (x)^T M(x) + M'_{+}(x) \big ] v. $$
The Riemannian metric is called contracting in $ K \subset G $ with exponent $ - \nu < 0 $, or a contraction metric on $K$, if
\begin{eqnarray}
\mathcal{L}_M(x) &\leq& - \nu \mbox{~for~all~} x \in K, \mbox{~where~} \label{contraction criteria}\\
\mathcal{L}_M(x) &:=& \max_{v^T M(x) v=1} L_M (x; v). \nonumber
\end{eqnarray}
\end{definition}

\begin{remark}
Fix $ x \in K$. Note that \eqref{contraction criteria} is equivalent to
\begin{equation*}
M(x) Df(x) + Df(x)^T M(x) + M'_{+}(x)  \preceq -2 \nu M(x)
\end{equation*}
where $ A \preceq B $ for $ A, B \in \Sb^{n\times n} $ means $ A - B $ is negative semi-definite, i.e.
$ w^T (A - B) w \leq 0 $ for all $ w \in \R^n$, see \cite[Remark 2.5]{contr2015giesl}.
\end{remark}

The next theorem will answer the question how contraction metrics can be used in our study of finding equilibria and their basin of attraction.

\begin{theorem}[Existence and uniqueness of the equilibrium]\label{th:1} \hspace*{2cm}\\
Let $\emptyset\not= K\subset \R^n$ be a compact, connected and
positively invariant set and $M$ be a Riemannian metric defined on a neighborhood $G$ of $K$ and contracting in $K$ with exponent $ - \nu <0$ as in Definition \ref{contraction metric}.
Then there exists one and only one equilibrium $x_0$ of system \eqref{eq_ode} in $K$;
$x_0$ is exponentially stable and $K$ is a subset of its basin of
attraction $\cA(x_0)$.
\end{theorem}

\begin{proof}
The proof is a mimic of \cite[Theorem 3.1]{contr2015giesl}, the only exception is that our Riemannian metric is not necessarily $C^1$ and we only assume the existence of the forward orbital derivative. This, however, does not change the structure of the proof. Thus, one can easily get the desired result by changing $M'$ to $M'_{+}$ there.  Note that LaSalle's principle, needed in the proof,
holds equally true for a $C^0$ mapping with forward orbital derivative which still fulfills the purpose of the theorem; see for example \cite[Theorem 2.20]{rbf2007giesl} and its proof and consider that a negative Dini derivative
implies that a function is decreasing.
\end{proof}

\begin{remark}
Let us define a linear differential operator $F$ associated with system \eqref{eq_ode}, given for any $ M\in C^{\tau}(G;\Sb^{n\times n})$ by
\begin{equation*}\label{F-def}
F(M)(x):= Df (x)^T  M(x)+M(x)Df(x)+M'_{+}(x).
\end{equation*}
As already mentioned in Remark \ref{OD/FOD}, when $\tau \geq 1$, the orbital derivative $M'$ exists and is equal to the positive orbital derivative $M'_{+}$. Therefore, in reading the following results one should not get confused using different references. However, we will prefer this notation as we need it later in the paper for functions with $\tau=0$.
\end{remark}

The following theorem which is a converse statement to Theorem \ref{th:1}, guarantees that within the basin of attraction for an  exponentially stable equilibrium of \eqref{eq_ode} there exists a contraction metric. Note that it provides a stronger smoothness property for $M$ than in Definition \ref{metric} and thus, it allows us to use the orbital derivative instead of the  forward orbital derivative (see Remark \ref{OD/FOD}). Note that on a compact subset $K \subset \cA(x_0)$, $M$ is a contraction metric by \eqref{contraction criteria}.

{\small
\begin{theorem}[{\small Existence and uniqueness of the contraction metric}] \cite[Theorems 2.2, 2.3]{giesl-cont-1}\label{un}
{\normalsize
Let $f\in C^s(\R^n;\R^n)$, $s\ge 2$. Let $x_0$ be an exponentially stable equilibrium of $\dot{x}=f(x)$ with basin of attraction $\cA(x_0)$.  Let $C\in C^{s-1}(\cA(x_0);\Sb^{n\times n})$ such that $C(x)$ is a positive definite matrix for all $x\in \cA(x_0)$. Then the matrix equation
\begin{equation} \label{matrixeq2}
F(M)(x) =  Df (x)^T  M(x)+M(x)Df(x)+M'_{+}(x)=-C(x)\ \ \ \ \ \text{ for all }x\in \cA(x_0)
\end{equation}
has a unique solution.

 In particular, $M\in C^{s-1}(\cA(x_0); \Sb^{n\times n})$, $M(x)$ is positive definite for all $x\in \cA(x_0)$ and
$M$ is of the form
$$ M(x)=\int_0^\infty \phi(\tau,0;x)^T
C(S_\tau x) \phi(\tau,0;x)\,d\tau, $$
 where $\tau \mapsto \phi(\tau,0;x)$ is the principal fundamental matrix solution to $\dot{y}=Df(S_t x)y$. 
 }
\end{theorem}
}

We will now recall some norm-related definitions and inequalities that will be used throughout the paper.
For an $A\in\R^{n\times n}$ define
\begin{eqnarray*}
\|A\|_{\max}&:=&\max_{i,j=1, 2, \ldots , n}|a_{ij}|\quad \\
\quad \|A\|_p&:=&\max_{\|x\|_p=1}\|A x\|_p\  \ \text{for $p=1,2,\infty$,} \\
\norm{A}{F} &:=& \left( \sum_{i,j=1}^{n}    a_{ij} ^2 \right)^{\frac12}.
\end{eqnarray*}
The following relations are well known:
\begin{eqnarray} 
\|A\|_1&=&\max_{j=1,\ldots,n}\sum_{i=1}^n |A_{ij}|, ~ \|A\|_1  = \|A^T\|_\infty ,  \label{norm-equivalence1} \nonumber \\
\|A\|_{\max}&\le& \|A\|_2 \le n \, \|A\|_{\max},\ ~  \|A\|_2\le \sqrt{n}\|A\|_\infty ,   \label{norm-equivalence2}\\
\dfrac{1}{\sqrt{n}} \, \|A\|_1 &\le& \|A\|_2\le \sqrt{n} \, \|A\|_1, \label{norm-equivalence3} \nonumber\\
\norm{A}{2} \leq \norm{A}{F} &\leq&\sqrt{n} \norm{A}{2}. \label{norm-equivalence4} \nonumber
\end{eqnarray}
For a symmetric and positive definite $A$, the largest singular value $\lambda_{\max}$ of $A$, which equals $\|A\|_2$ and is the largest of its eigenvalues, is the smallest number such that
$ A \preceq \lambda_{\max} I $. \\
We recall that $ \|M\|_{L^\infty (K)}= \displaystyle \operatorname {ess} \sup_{x\in K}\|M(x)\|_2$ for any $K\subset \R^n$.  Further, if $M$ is continuous and a set $K\subset \R^n$ has the property, that every neighbourhood (in $K$) of every $x\in K$ has a strictly positive measure, then the essential supremum is identical to the supremum.

For a function $ W \in  C^k(\cD; \cR) $, where $\cD\subset \R^n$ is a non-empty open set, and $ \cR $ is  $ \R, \R^n, \Sb^{n\times n}, $ or $\R^{n \times n}$, we define the $C^k$-norm  as
\begin{equation}
\norm{W}{C^k(\cD;\cR) } := \sum_{| {\balpha} | \leq k} \, \sup_{x \in \cD} \norm{D^{\balpha} W(x)}{2}
\end{equation}
where $ \balpha\in\N_0^n $ is a multi-index and $ |\balpha| := \sum_{i=1}^{n} \alpha_i$.  Note that when all relevant $D^{\balpha} W$ can be continuously extended to $\overline{\cD}$, the $C^k$-norm is defined on $\overline{\cD}$ in a similar fashion.

\begin{remark}\label{function-norm-inequalities}
Throughout the paper we use the following inequalities regarding the function norms.  Assume that $ g \in C^2(\cD; \R)$, where $\emptyset \neq \cD\subset \R^n $ is open, and let $ K$ be a compact subset of $\cD$. Denote the Hessian of $g$ at $ x \in \cD $ by $ H_g(x) := \left( \frac{\partial^2 g}{\partial x_i \partial x_j}(x) \right)_{ij} $ and the upper bound
\begin{equation*}
 B_{K} := \max_{x \in K \atop i,j=1,\ldots,n }\left|\frac{\partial^2 g}{\partial x_i\partial x_j}(x)\right|
\end{equation*}
on all the second-order derivatives of $g$ on $K$ by $B_{K}$.

The first inequality bounds $B_{K}$ by the $C^2$-norm of $g$.
\begin{eqnarray}\label{B-C2-norm}
 \norm{g}{C^2(\cD; \R)} & := &  \sum_{| {\balpha} | \leq 2} \, \sup_{x \in \cD} \norm{D^{\balpha} g(x)}{2} \geq  \sum_{| {\balpha} | = 2} \, \sup_{x \in \cD} \norm{D^{\balpha} g(x)}{2} \nonumber \\
& \geq &  \sum_{i,j = 1}^{n} \max_{x \in K}\left|\frac{\partial^2 g}{\partial x_i\partial x_j}(x)\right|
 \geq \max_{i,j=1,\ldots,n } \max_{x \in K} \left|\frac{\partial^2 g}{\partial x_i\partial x_j}(x)\right| =  B_{K}.
\end{eqnarray}
The second inequality bounds the $2$-norm of the Hessian of $g$ by its $C^2$-norm.
 This is a sharper estimate than \cite[Lemma 4.2]{baier2012linear}.
\begin{eqnarray}\label{Hessian-C2-norm}
\sup_{x \in \cD}\,  \norm{H_g (x)}{2} & \leq & \sup_{x \in \cD} \, \norm{H_g(x)}{F} = \sup_{x \in \cD} \left( \sum_{i,j=1}^{n}   \left( \frac{\partial^2 g}{\partial x_i \partial x_j}(x)\right)^2 \right)^{\frac12}  \nonumber \\
& \leq & \sup_{x \in \cD} \,  \sum_{i,j=1}^{n}    \left| \frac{\partial^2 g}{\partial x_i \partial x_j}(x) \right|  \le  \sum_{| {\balpha} | = 2} \sup_{x \in \cD} \norm{D^{\balpha} g(x)}{2}
 \leq  \norm{g}{C^2(\cD; \R)}.
\end{eqnarray}
The third inequality bounds the  $2$-norm of the derivative $Dg(x) = \left( \dfrac{\partial g}{\partial x_i}(x) \right)_{i} $ by the  $C^1$-norm of $g$.
\begin{eqnarray}\label{D-C1-norm}
\sup_{x \in \cD} \norm{Dg(x)}{2} & = & \sup_{x \in \cD} \left( \sum_{i=1}^{n} \left( \dfrac{\partial g}{\partial x_i}(x) \right)^2 \right)^{1/2}  \leq  \sup_{x \in \cD} \sum_{i=1}^{n} \left| \dfrac{\partial g}{\partial x_i}(x) \right|  \nonumber \\
& \le & \sum_{| {\balpha} | = 1} \sup_{x \in \cD} \norm{D^{\balpha} g(x)}{2} \leq \norm{g}{C^1(\cD; \R)}.
\end{eqnarray}

Now, for a vector-valued function $g \in C^2(\cD; \R^n)$ one can extend the definition of $B_K$ using the formula
$$ B_{K} := \max_{x \in K \atop i, j, k = 1, \ldots, n } \left| \frac{\partial^2 g_k }{\partial x_i\partial x_j} (x) \right|$$
and prove inequalities analogous to \eqref{B-C2-norm} and \eqref{D-C1-norm} in a similar way,
\begin{eqnarray*}\label{B-C2-norm-Rn}
 \norm{g}{C^2(\cD; \R^n)}  \geq 
  \sum_{i,j = 1}^{n} \max_{x \in K} \left( \sum_{k=1}^{n} \left( \frac{\partial^2 g_k }{\partial x_i\partial x_j} (x) \right)^2 \right)^{\frac12} \geq B_K
\end{eqnarray*}
and, noting that $Dg(x)$ is a matrix,
\begin{eqnarray}\label{D-C1-norm-Rn}
\sup_{x \in \cD} \norm{Dg(x)}{2} & = & \sup_{x \in \cD} \norm{ \left( \frac{\partial g_k }{\partial x_i } (x)\right)_{ik} }{2} \leq \sup_{x \in \cD} \left( \sum_{i,k=1}^{n} \left(  \frac{\partial g_k }{\partial x_i } (x)  \right)^2 \right)^{\frac12}   \\
& \leq & \sup_{x \in \cD} \sum_{i=1}^{n} \left(  \sum_{k=1}^{n} \left(  \frac{\partial g_k }{\partial x_i } (x)  \right)^2 \right)^{\frac12} \le
\sum_{| {\balpha} | = 1} \, \sup_{x \in \cD} \norm{D^{\balpha} g(x)}{2}  \nonumber \\
& \le &  \norm{g}{C^1(\cD; \R^n)}.\nonumber
\end{eqnarray}
Further, for $g \in C^3(\cD; \R^n)$ and with
$$ B_{3,K} := \max_{x \in K \atop i, j, k, \ell = 1, \ldots, n } \left| \frac{\partial^3 g_\ell }{\partial x_i \partial x_j \partial x_k} (x) \right|, $$
we can analogously prove that
\begin{equation*}
B_{3,K} \leq \sum_{i,j,k=1}^n \max_{x\in K}  \left( \sum_{\ell=1}^{n} \left( \frac{\partial^3 g_\ell }{\partial x_i\partial x_j\partial_k} (x) \right)^2 \right)^{\frac12} \le \norm{g}{C^3(\cD; \R^n)}.
\end{equation*}

Finally, we show that the $C^2$-norm of each component $W_{ij} \in C^2(\cD; \R) $ of  a matrix-valued function $W \in C^2(\cD; \R^{n \times n}) $ is bounded by the $C^2$-norm
of $W$.  First, note that
\begin{eqnarray*}
\norm{W}{C^2(\cD; \R^{n \times n})} & := & \sum_{| {\balpha} | \leq 2} \, \sup_{x \in \cD} \norm{D^{\balpha} W(x)}{2} \\
 & = & \sup_{x \in \cD} \norm{W(x}{2} +  \sum_{k=1}^{n} \sup_{x \in \cD} \norm{\left( \frac{\partial W_{ij} }{\partial x_k} (x)\right)_{ij}}{2} +   \sum_{k, \ell = 1}^{n} \sup_{x \in \cD}\norm{\left( \frac{\partial^2 W_{ij} }{\partial x_k \partial x_\ell} (x)\right)_{ij}}{2}.
\end{eqnarray*}
Now, using the $2$-norm definition we obtain
\begin{small}
\begin{eqnarray*}
\norm{W(x)}{2}&=&\max_{u\in \R^n \atop \norm{u}{2} = 1} \norm{W(x) u}{2} =  \max_{u\in \R^n \atop \norm{u}{2} = 1} \left( \sum_{i=1}^{n} \left( \sum_{j=1}^{n} W_{ij}(x)u_j \right)^2 \right)^{\frac12} \\
& \geq & \max_{u\in \R^n \atop \norm{u}{2} = 1}\left| W_{i^*j^*}(x) u_{j^*}  \right| \ge   \norm{W_{i^*j^*}(x)}{2},
\end{eqnarray*}
\end{small}
where in the last line $i^*, j^* $ are arbitrary fixed indices in the range $1,2, \ldots, n $, and  similarly
$$
\norm{\left( \frac{\partial W_{ij} }{\partial x_k} (x)\right)_{ij}  }{2} \geq \ \norm{\frac{\partial W_{i^*j^*} }{\partial x_k}(x)}{2}\ \ \text{and}\ \
\norm{\left( \frac{\partial^2 W_{ij}}{\partial x_k \partial x_\ell}(x) \right)_{ij}}{2} \geq    \norm{ \frac{\partial^2 W_{i^*j^*} }{\partial x_k \partial x_\ell}(x)}{2}.
$$
Together this yields that for any $i, j = 1, 2, \ldots, n $ we have
\begin{equation}\label{matrix-to-component-norm}
\norm{W}{C^2(\cD; \R^{n \times n})} \geq \norm{W_{ij}}{C^2(\cD; \R)}.
\end{equation}
\end{remark}

The last statement of this section is a powerful result that we will use later to obtain the error estimates of approximated maps; note that \eqref{estim} follows directly from the proof in \cite[Theorem 2.4]{giesl-cont-1}.

\begin{theorem}[Perturbation effect on contraction metrics]\label{est}  \cite[Theorem 2.4]{giesl-cont-1}

Let $f\in C^s(\R^n;\R^n)$, $s\ge 2$. Let $x_0$ be an exponentially stable equilibrium of $\dot{x}=f(x)$ with basin of attraction $\cA(x_0)$.  Let $C_i\in C^{s-1}(\cA(x_0);\Sb^{n\times n})$, $i=1,2$, such that $C_i(x)$ is a positive definite matrix for all $x\in \cA(x_0)$.

Let $M_i\in C^{s-1}(\cA(x_0); \Sb^{n\times n})$
 be the unique solution (see Theorem \ref{un}) of the matrix equation
\begin{eqnarray*}\label{matrixeq3}
F(M_i)(x) := Df(x)^T M_i(x)+M_i(x)Df(x)+(M_{i})'_+(x)&=&-C_i(x)
\end{eqnarray*}
for all $ x \in \cA(x_0) $, where $i=1,2$. Let $K\subset \cA(x_0)$ be a compact set.

Then there is a constant $\alpha$, independent of $M_i$ and $C_i$ such that
\begin{equation*}
\sup_{x \in K} \norm{M_1(x)-M_2(x)}{2} \leq \, \alpha \sup_{x \in \overline{\gamma^+(K)}} \norm{C_1(x)-C_2(x)}{2},\label{estim}
\end{equation*}
where $\gamma^+(K)=\bigcup_{t\ge 0}S_tK$.
\end{theorem}

The theorem shows that if $\|F(M)(x) - F(S)(x)\| \leq \epsilon$ for all
$ x \in \overline{\gamma^+(K)} $,
then we have
$ \| M(x) - S(x) \| \leq \alpha \, \epsilon $ for all $ x \in K$. In particular, as $M$ is positive definite in $K$, so is $S$ for all small enough $\epsilon>0$. Note that for a positively invariant and compact set $K$ we have $\overline{\gamma^+(K)} = K $.

In the rest of the paper we will consider the PDE \eqref{matrixeq2} with a constant right hand side
\begin{eqnarray}\label{matrixeq4}
Df (x)^T M(x)+M(x)Df(x)+M'_{+}(x) = -C,
\end{eqnarray}
that is, $ F(M)(x) =-C $ for all $x\in \cA(x_0) $.

\section{First Approximation using RBF} \label{RBF-section}

In this section we introduce the proper setting for the optimal recovery problem and then review two theorems: one regarding the existence and uniqueness of the optimal recovery (Theorem \ref{th:S}) and an error estimate for the approximation (Theorem \ref{thmfinal}).

Let $\Omega\subset\R^n$ be a domain and $\sigma>n/2$ be given. Then, the
matrix-valued Sobolev space $H^\sigma(\Omega;\R^{n\times n})$ consists of
all matrix-valued functions $M$ having each component $M_{ij}:\Omega\to \R$ in the Sobolev space
$H^\sigma(\Omega)$.
Similarly, the
Sobolev space $H^\sigma(\Omega;\Sb^{n\times n})$ consists of
all symmetric matrix-valued functions $M$ having each component $M_{ij}$ in
$H^\sigma(\Omega)$.

$H^\sigma(\Omega;\R^{n\times n})$ and $H^\sigma(\Omega;\Sb^{n\times n})$ are Hilbert spaces with inner
product given by
\[
\langle M, S\rangle_{H^\sigma(\Omega;\R^{n\times n})}:=\sum_{i,j=1}^n \langle M_{ij},S_{ij}\rangle_{H^\sigma(\Omega)},
\]
where $\langle \cdot,\cdot\rangle_{H^\sigma(\Omega)}$ is the usual inner product on $H^\sigma(\Omega)$;
the same inner product can be used for $H^\sigma(\Omega;\Sb^{n\times n})$. They are also \textit{reproducing kernel Hilbert spaces} (RKHS). In the following we assume that $W$ is either $\R^{n\times n}$ or its subspace $\Sb^{n\times n}$.
On $W$ we define the  inner product
\begin{eqnarray*}
\langle \alpha,\beta \rangle_W &=&\sum_{i,j=1}^n \alpha_{ij}\beta_{ij},
\qquad \alpha=(\alpha_{ij}), \beta=(\beta_{ij}),
\label{W}
\end{eqnarray*}
which renders it a Hilbert space.  We denote by $\cL(W)$ the linear space of all linear and bounded
operators $W \to W$.

\begin{definition}[Reproducing Kernel Hilbert Space] \label{def:RKHS}
A Hilbert space $\cH(\Omega;W)$ of functions $ f: \Omega \to W $ is called  {\em
    reproducing kernel Hilbert space} if there is a function
  $\Phi:\Omega\times\Omega\to \cL(W)$  with the following properties{\rm\,:}
\begin{enumerate}[nosep]
\item $\Phi(\cdot,x)\alpha \in \cH(\Omega;W)$ for all $x\in\Omega$ and
  all $\alpha\in W$.
\item $\langle f(x),\alpha\rangle_W = \langle f,
  \Phi(\cdot,x)\alpha\rangle_{\cH}$ for all $f\in\cH(\Omega;W)$, all
  $x\in\Omega$ and all $\alpha\in W$.
\end{enumerate}
The function $\Phi$ is called a {\em reproducing kernel} of
$\cH(\Omega;W)$.
\end{definition}

A kernel $\Phi$ is thus a mapping $\Phi:\Omega\times \Omega\to
\cL(W)$, $W=\R^{n\times n}$ or $W=\Sb^{n\times n}\subset\R^{n\times n}$,  and can be represented by a tensor of order $4$, i.e.~we
will write
$
\Phi=(\Phi_{ijk\ell})
$
and define its action on $\alpha\in \R^{n\times n}$ by
\begin{eqnarray*}
(\Phi(x,y)\alpha)_{ij} &=& \sum_{k,\ell=1}^n
\Phi(x,y)_{ijk\ell}\alpha_{k\ell}.\label{action}
\end{eqnarray*}

\begin{lemma}[Induction of reproducing kernels] \label{kernel1} \cite[Lemma 3.2]{giesl2018kernel}

Let $\Omega\subset\R^n$ be a domain and $\sigma>n/2$ be
  given. Assume that $\phi:\Omega\times\Omega\to\R$ is the reproducing
  kernel of $H^\sigma(\Omega)$. Then, $H^\sigma(\Omega;\R^{n\times n})$ and $H^\sigma(\Omega;\Sb^{n\times n})$ are also reproducing kernel Hilbert spaces with the reproducing kernel
  $\Phi$ defined by
\begin{equation}\label{phi-tensor}
\Phi(x,y)_{ijk \ell}:=\phi(x,y) \delta_{ik} \delta_{j \ell}
\end{equation}
for $ x, y \in \Omega $ and $ 1 \le i, j, k, \ell \le n$.
\end{lemma}

It is now time to introduce the problem: how to recover a function with values in $W$ given only finitely many information of it.

\begin{definition}[Optimal recovery of a function]\label{optrec}
Given $N$ linearly independent functionals
$\lambda_1,\ldots,\lambda_N\in \cH(\Omega;W)^*$ of a reproducing kernel Hilbert space $\cH(\Omega;W)$ and $N$ values
$r_1=\lambda_1(M),\ldots, r_N=\lambda_N(M)\in\R$ generated by an
element $M\in\cH(\Omega;W)$. The optimal recovery of $M$ based on this
information is defined to be the element $ S \in\cH(\Omega;W)$ which
solves
\[
\min \left\{\| s \|_\cH : s\in\cH(\Omega;W) \mbox{ with
}\lambda_j(s)=r_j, 1\le j\le N\right\}.
\]
\end{definition}

We choose Wendland functions as the radial basis functions, which will define the reproducing kernel $\Phi$ needed for our optimal recovery problem. For more details on these functions and their properties, see \cite{Wendland1998}.

\begin{definition}[Wendland functions, \cite{Wendland1998}]  \label{defwend}
Let $l\in\mathbb N$, $k\in \mathbb N_0$. We define by recursion
\begin{eqnarray*}
\psi_{l,0}(r)&=&(1-r)^l_+,
\label{rec0}\\
\mbox{and } \psi_{l,k+1}(r)&=&\int_{r}^1
t\psi_{l,k}(t)\,dt\label{reci}
\end{eqnarray*}
for $r\in \mathbb R_0^+$.
Here we set $x_+=x$ for $x\ge 0$, $x_+=0$ for
$x<0$, and $x_+^l:=(x_+)^l$.

With $ l:=\lfloor \frac{n}{2}\rfloor+k+1$ the function $\Phi(x):=\psi_{l,k}(c \|x\|_2)$  belongs to $C^{2k}(\mathbb R^n)$ for any $c>0$ and the reproducing kernel Hilbert space with reproducing kernel $\Phi$ given by a Wendland function is norm-equivalent to the Sobolev space $H^{\sigma}(\Omega)$, where $\sigma=k+\frac{n+1}{2}$.
\end{definition}

Now consider $H^\sigma(\Omega; \Sb^{n\times n})$, the matrix-valued Sobolev space with reproducing kernel $\Phi:\Omega\times\Omega\to\cL(\Sb^{n\times n})$
as in \eqref{phi-tensor}, $\phi(x,y)=\psi_{l,k}(c \|x-y\|_2)$, where $\psi_{l,k}$ is a Wendland function with $ l:=\lfloor \frac{n}{2}\rfloor+k+1$ and $c>0$.   We again have  $\sigma=k+\frac{n+1}{2}$.

We then define the linear functionals
$\lambda_k^{(i,j)}:H^\sigma(\Omega;\Sb^{n\times n})\to\R$ by
\begin{eqnarray}
\lambda_{k}^{(i,j)}(M) &=& e_i^T \left[Df^T(x_k)M(x_k)
  +M(x_k)Df(x_k)+M'_{+}(x_k) \right]e_j\label{functionals}\\
&=:& e_i^T F_k(M)e_j\nonumber\\
&=& e_i^T F(M)(x_k)e_j\nonumber
\end{eqnarray}
for $x_k\in\Omega$, $1\le k\le N$ and $1\le i\le j\le n$. Here, $e_i$
denotes the usual $i$th unit vector in $\R^n$.  Thus $\lambda_{k}^{(i,j)}(M)$ is simply the $(i,j)$th element of the matrix $F(M)(x_k)$.

We define $ E_{\mu \mu}^s $ to be the matrix with value $1$ at position $(\mu, \mu)$ and value zero everywhere else. For $ \mu < \nu $, we define $ E_{\mu \nu}^s $ to be the matrix with value
$ 1/ \sqrt{2} $ at positions $(\mu, \nu)$ and $(\nu, \mu)$ and value zero everywhere else. It is easy to see that
$ \{  E_{\mu \nu}^s : 1\leq \mu \leq \nu \leq n \}$  is an orthonormal basis of $\Sb^{n\times n}$.  We also define $ E_{\mu \nu} \in \R^{n \times n}$ to be the matrix with value 1 at position
$(\mu, \nu)$ and value zero everywhere else.

We can compute the solution $S$ of the optimal recovery problem
as in Definition \ref{optrec}. This gives the following result with our notation:

\begin{theorem}[{\small Existence and uniqueness of the optimal recovery}] \cite[Theorem 5.2]{giesl2018kernel}\label{th:S}

Let $\sigma>n/2+1$ and let $\Phi:\Omega\times\Omega\to\cL(\Sb^{n\times
  n})$ be the reproducing kernel of $H^\sigma(\Omega;\Sb^{n\times n})$.
Let $X=\{x_1,\ldots,x_N\}\subset\Omega$ be pairwise distinct points and let
$\lambda_k^{(i,j)}\in H^\sigma(\Omega;\Sb^{n\times n})^*$, $1\le k\le N$ and $1\le
i,j\le n$ be defined by (\ref{functionals}).
Then there is a unique  function $S\in H^\sigma(\Omega;\Sb^{n\times n})$ solving
\[
\min\left\{ \|S\|_{H^\sigma(\Omega; \Sb^{n\times n})} : \lambda_{k}^{(i,j)}(S) = -C_{ij}, 1\le
i\le j\le n, 1\le k\le N\right\},
\]
where $C=(C_{ij})_{i,j=1,\ldots,n}$ is a symmetric, positive definite matrix.
It has the form
\begin{eqnarray} \label{Sformula}
S(x) &=& \sum_{k=1}^N\sum_{1\le i\le j\le n} \gamma_k^{(i,j)}
\sum_{1\le \mu\le \nu\le n}\lambda_{k}^{(i,j)}(\Phi(\cdot,x)E^s_{\mu\nu})E^s_{\mu\nu} \nonumber\\
&=& \sum_{k=1}^N\sum_{1\le i\le j\le n}
\gamma_k^{(i,j)}
\bigg[
\sum_{ \mu=1}^n
F_k(\Phi(\cdot,x)_{\cdot,\cdot,\mu,\mu})_{ij}E_{\mu\mu}\nonumber\\
&&+
\frac{1}{2}
\sum_{\substack{ \mu,\nu=1\\\mu\not =\nu}}^n
[F_k(\Phi(\cdot,x)_{\cdot,\cdot,\mu,\nu})_{ij}+
F_k(\Phi(\cdot,x)_{\cdot,\cdot,\nu,\mu})_{ij}]E_{\mu\nu}\bigg],\label{form1}
\end{eqnarray}
where the coefficients $\gamma_k=(\gamma_k^{(i,j)})_{1\le i\le j\le n}$
are determined by substituting \eqref{Sformula} in the operator equations $\lambda_\ell^{(i,j)}(S)=-C_{ij}$ for $1\le i\le j\le n$, $1\le \ell\le N$.

If the kernel $\Phi$ is given by \eqref{phi-tensor} then we also have the alternative expression
\begin{eqnarray*}
S(x) &=&
\sum_{k=1}^N\sum_{i,j=1}^n\beta_k^{(i,j)} \sum_{\mu,\nu=1}^n
F_k(\Phi(\cdot,x)_{\cdot,\cdot,\mu,\nu})_{ij} E_{\mu\nu}\label{form2}
\end{eqnarray*}
where the symmetric matrices $\beta_k\in \Sb^{n\times n}$ are defined by  $\beta_k^{(j,i)}=\beta_k^{(i,j)}=\frac{1}{2}\gamma_k^{(i,j)}$ if $i\not= j$ and
$\beta_k^{(i,i)}=\gamma_k^{(i,i)}$.
\end{theorem}

We will measure the error of the optimal recovery in terms of
the so-called fill distance or mesh norm
\[
h_{X,\Omega}:=\sup_{x\in\Omega}\min_{x_i \in X}\|x-x_i\|_2.
\]

\begin{theorem}[Error estimates for the RBF approximation] \cite[Theorem 5.3]{giesl2018kernel} \label{thmfinal}

Let $f\in C^{\sigma + 1}(\R^n;\R^n)$, with $ \sigma \in \N $ and $\sigma > n/2 + 1$. Assume that $x_0$ is an
exponentially stable equilibrium of $\dot x =f(x)$ with basin of
attraction $\cA(x_0)$. Let $C\in\Sb^{n\times n}$ be a positive definite
(constant) matrix and let $M\in C^{\sigma}(\cA(x_0);\Sb^{n\times n})$ be
the solution of the PDE (\ref{matrixeq4}) from Theorem \ref{un}. Let
$K\subset \Omega\subset \cA(x_0)$ be a positively invariant and
compact set, where $\Omega$ is open with Lipschitz boundary. Finally, let
$S$ be the optimal recovery of $M$ from Theorem \ref{th:S}. Then, we have the error
estimate
\begin{eqnarray}
\sup_{x \in K} \norm{M(x)-S(x)}{2} \le
  \alpha \|F(M)-F(S)\|_{L^\infty(\Omega;\Sb^{n\times n})} \le \beta
  h_{X,\Omega}^{\sigma-1-n/2}
\|M\|_{H^\sigma(\Omega;\Sb^{n\times n})}
\label{RBF-est}
\end{eqnarray}
for all sets $X\subset\Omega$ with sufficiently small $h_{X,\Omega}$.
\end{theorem}

Note that in this theorem $\alpha, \beta $ are positive constants independent of $M, S, $ and $ X $; the statement in \eqref{RBF-est} follows directly from the proof
of \cite[Theorem 5.3]{giesl2018kernel}.
The theorem indicates that $S$, itself, is a contraction metric in $K$ provided $h_{X,\Omega}$ is sufficiently small.

\begin{remark} \label{norm-remark}
It is worth mentioning another useful norm estimate for $S \in H^\sigma(\Omega;\Sb^{n\times n}) $, the optimal recovery of $M$ from Theorem \ref{th:S}, over $ \cD \subset \Omega$, a bounded open subset of $\R^n$ with $C^1$ boundary. Let $k\ge 2$ if $n$ is odd and $k\ge 3$ if $n$ is even. Let $S$ be the approximation of $M$, using the Wendland function $\psi_{l,k}$ with $l=\lfloor \frac{n}{2}\rfloor +k+1$ and the collocation points $X$. Note that
the reproducing kernel Hilbert space $\cH(\Omega;\Sb^{n\times n})$ with reproducing kernel $\Phi$ given by the Wendland function is norm-equivalent to the Sobolev space $H^{\sigma}(\Omega;\Sb^{n\times n})$, where $\sigma=k+\frac{n+1}{2}$. Then we have
\begin{equation} \label{norm-remark-eq}
\norm{S}{C^{2}(\overline{\cD};\Sb^{n\times n})} \leq \zeta \norm{M}{H^\sigma(\Omega;\Sb^{n\times n})},
\end{equation}
where $ \zeta > 0 $ is a constant independent of the collocation points $X$, and the approximation $S$. The inequality is proved using that the approximation $S$ is norm-minimal, that is, $\norm{S}{\cH(\Omega;\Sb^{n\times n})} \leq \norm{M}{\cH(\Omega;\Sb^{n\times n})} $; for more details see \cite[Lemma 3.8]{GiHa2015combi}.
\end{remark}

\section{Second Approximation using CPA Interpolation} \label{CPA-section}
In this section we will first provide the necessary definitions and statements about the triangulations and continuous piecewise affine interpolations of a function. Then we consider a verification problem to check whether our criteria for a contraction metric are fulfilled by the interpolated function, and finally we derive error estimates for this process.


\begin{definition}[simplex]
Given vectors $x_0,x_1,\ldots,x_n\in\R^n$ that are affinely independent, i.e.~the vectors $x_1-x_0,x_2-x_0,\ldots,x_n-x_0$ are linearly independent, the convex hull%
\begin{equation*}
  \T = \co(x_0,x_1,\ldots,x_n) := \left\{\sum_{k=0}^n \lambda_k x_k\,:\, \lambda_k\in[0,1]\ \text{and}\ \sum_{k=0}^n\lambda_k=1\right\}%
\end{equation*}
is called an $n$-simplex or simply a simplex.
A set%
\begin{equation*}
  \co(x_{k_0},x_{k_1},\ldots,x_{k_j}) := \left\{\sum_{i=0}^j \lambda_{k_i} x_{k_i}\,:\, \lambda_{k_i}\in[0,1]\ \text{and}\ \sum_{i=0}^j\lambda_{k_i}=1\right\}%
\end{equation*}
with $0\le k_0< k_1 < \ldots <k_j \le n$ and $0\le j <n$ is called a $j$-face of the simplex $\T$.%
\end{definition}

\begin{definition}[Triangulation]\label{scdef}
We call a finite set $\cT=\{\T_\nu\}_\nu$ of $n$-simplices $\T_\nu$ a triangulation in $\R^n$, if two simplices $\T_\nu,\T_\mu\in\cT$, $\mu\neq \nu$, intersect in a common face or not at all.
For a triangulation $\cT$ we define its \emph{domain} and \emph{vertex set} as
$$
\cD_\cT:=\bigcup_{\T_\nu}\T_\nu\ \ \text{and}\ \ \cV_\cT := \{\bx\in\R^n \,:\, \bx\ \text{is a vertex of a simplex in $ \cT$ }\}.
$$
We also say that $\cT$ is a triangulation of the set $\cD_\cT$.\\

For a triangulation $\cT=\{\T_\nu\}$ and constants $ h, d > 0 $, we say that
$\cT$ is \emph{$(h, d)$-bounded} if it fulfills the following conditions:
\begin{itemize}
\item[(i)]
The \textit{diameter} of every simplex $ \T_\nu \in \cT $ is bounded by $h$, that is
\[ h_\nu := \mbox{diam}(\T_\nu) := \max_{x, y \in \T_\nu } \| x - y \|_2 < h. \]
\item[(ii)]
The \textit{degeneracy} of every simplex $ \T_\nu \in \cT $ is bounded by $d$ in the sense that
\[ h_\nu \| X_\nu ^{-1} \|_1 \le d, \]
where $X_\nu := (x_1 ^\nu - x_0 ^\nu, x_2 ^\nu - x_0 ^\nu, \cdots, x_n ^\nu - x_0 ^\nu)^T $ is the so-called \emph{shape matrix} of the simplex $\T_\nu$.
\end{itemize}
\end{definition}

\begin{definition}[CPA interpolation]\label{CPA-def}
Let $\cT$ be a triangulation in $\R^n$  and assume some values $\widetilde{P}_{ij}(x_k)\in\R$ are fixed for every $x_k\in\cV_\cT$ and every $i,j=1,2,\ldots,n$. Then we can uniquely construct a continuous function $P:\cD_\cT\to\R^{n\times n}$, that is affine on each simplex $\T_\nu\in\cT$ in the following way{\,\rm:}
An $x \in \T_\nu  = \co(x_0,\ldots,x_n)$ can be written uniquely as $x=\sum_{k=0}^{n}\lambda_k x_k$ with $\lambda_k\in [0,1]$ and $\sum_{k=0}^{n}\lambda_k=1$ and we define%
\begin{equation*}
  P_{ij}(x) := \sum_{k=0}^{n}\lambda_k \widetilde P_{ij}(x_k)%
\end{equation*}
and%
\begin{equation*}
  P(x) := \begin{pmatrix}
       P_{11}(x) & P_{12}(x)& \cdots & P_{1n}(x) \\
       P_{21}(x) & P_{22}(x)&   \cdots & P_{2n}(x) \\
       \vdots & \vdots & \ddots & \vdots \\
       P_{n1}(x) & P_{n2}(x) & \cdots & P_{nn}(x) \\
     \end{pmatrix}.
\end{equation*}
We refer to the functions $P_{ij}$ and $P$ as the CPA interpolations of the values $\widetilde P_{ij}(x_k)$ and $\widetilde P(x_k)=(\widetilde P_{ij}(x_k))_{i,j=1,\ldots,n}$, respectively.

%

Then we can uniquely define continuous functions
$P_{ij}:\cD_\cT\to\R$ through\,{\rm :}
\begin{enumerate}
  \item[\textit{(i)}]   $P_{ij}(x):=\widetilde P_{ij}(x)$ for every $x\in\cV_\cT$,
  \item[\textit{(ii)}]  $P_{ij}$ is affine on every simplex $\fS_\nu\in\cT$, i.e. there is a vector $w_{ij}^\nu\in\R^n$ and a number $b_{ij}^\nu\in\R$, such that
  $$P_{ij}(x)=(w_{ij}^\nu)^T x+b_{ij}^\nu$$ for all $ x\in \fS_\nu$.
\end{enumerate}
The set of all such continuous and piecewise affine functions $\cD_\cT\to\R$ fulfilling (i) and (ii) is denoted by $\CPA[\cT]$.

Note that for every simplex $\fS_\nu\in\cT$ we have $ \nabla P_{ij}|_{\fS_\nu^\circ} =w_{ij}^\nu$, where $w_{ij}^\nu\in\R^n$ is as in (ii).   \\
Assume $W$ is a matrix-valued function defined on $ \cD_\cT $, fix the values $ \widetilde P(x_k) = W(x_k) $ for every vertex $ x_k \in \cV_\cT $, and continue the procedure mentioned above to create a continuous piecewise affine function $P$. Then we call $P$ the CPA interpolation of the function $W$ on $\cT$.

Note that if $\widetilde{P}(x_k)\in \Sb^{n\times n}$ for all $ x_k \in \cV_\cT $, then $P\colon  \cD_\cT \to\Sb^{n\times n}$.
\end{definition}

\begin{remark}[Orbital derivative]\label{basic}
Let $P(x)$ be as in Definition \ref{CPA-def} and fix a point $x\in\cD^\circ_\cT$.  As shown in the proof of \cite[Lemma 4.7]{GiHa2013CPAmetric}, there exists a $\T_\nu=\co(x_0,\ldots,x_n)\in \cT$ and a number $\theta^*>0$ such that $x+\theta f(x)\in \T_\nu$ for all $\theta\in [0,\theta^*]$.
Then the forward orbital derivative $(P_{ij})'_+(x)$ defined by formula \eqref{fod} (see Remark \ref{for-orb-der-formula}), is given by
\begin{equation*}
  (P_{ij})'_+(x) = w_{ij}^\nu \cdot f(x),%
\end{equation*}
where $w_{ij}^\nu$ was defined in Definition \ref{CPA-def}.
\end{remark}


\begin{lemma}[Error estimates for CPA interpolation] \label{CPA-estimate} 
Let $ \cT = \{\T_\nu\}$ be an $ (h, d) $-bounded triangulation in $\R^n$ and let $ \cD\supset \cD_\cT$ be an open set.  Assume that $ W \in  C^2(\cD;\R^{n \times n}) $ with $\norm{W}{C^2(\cD;\R^{n \times n}) }<\infty$  and define
\begin{equation*}
\gamma :=  1 + \dfrac{d n^{3/2}}{2} .
\end{equation*}
Denote by $W_C$ the CPA interpolation of $ W$ on $\cT$. Then the following
estimates hold true for all $1\le i,j\le n$\,{\rm :}
\begin{eqnarray}
\| W_C (x) - W(x) \|_2 &\le & n h^2  \norm{W}{C^2(\cD;\R^{n \times n})}\ \   \text{for all $x \in \cD_\cT$,} \label{es-eq1}\\
\|\nabla (W_C)^\nu_{ij} - \nabla W_{ij}(x) \|_{1}& \le&  h \gamma \norm{W}{C^2(\cD;\R^{n \times n})}\ \  \text{for all $\T_\nu \in \cT$ and all $x \in \T_\nu$,} \label{es-eq2}\\
\|\nabla (W_C)^\nu_{ij}  \|_{1} &\le& ( 1 +  h \gamma)  \norm{W}{C^2(\cD;\R^{n \times n})} \ \ \text{for all $\T_\nu \in \cT$,}
 \label{es-eq3}
\end{eqnarray}
\end{lemma}

\begin{proof}
This lemma is a counterpart of \cite[Lemma 4.15]{GiHa2015combi} for matrix-valued functions. We just prove inequality \eqref{es-eq1}, where we have obtained a sharper estimate. Observe that
\begin{equation*}
\| W_C (x) - W(x) \|_{\text{max}} :=  \max_{i, j = 1, 2, \ldots, n} \big| (W_{C})_{ij} (x) - W_{ij}(x) \big|,
\end{equation*}
in which $ W_{ij} \in C^2(\cD;\R) $ and $W_{ij}, (W_{C})_{ij}$ are the components of $W$ and $W_C$, respectively. Now we can use the ideas of \cite[Lemma 4.15]{GiHa2015combi} and inequality \eqref{Hessian-C2-norm} to obtain
\begin{equation*}
\big| (W_{C})_{ij} (x) - W_{ij}(x) \big| \leq h^2 \, \max_{z \in \cD_{\cT}} \norm{H_{W_{ij}}(z)}{2}
\leq h^2 \norm{W_{ij}}{C^2(\cD;\R)},
\end{equation*}
where $ H_{W_{ij}}(z) $ denotes the Hessian of $W_{ij}$ at $z$. Considering inequality \eqref{matrix-to-component-norm} of Remark \ref{function-norm-inequalities} yields that
\begin{equation*}
\| W_C (x) - W(x) \|_{\text{max}} 
~  \leq ~\max_{i, j = 1, 2, \ldots, n} h^2 \, \| W_{ij} \|_{C^2(\cD;\R)} 
\end{equation*}
for all $x\in \cD_\cT$.
It only remains to consider norm equivalence relations \eqref{norm-equivalence2} \eqref{matrix-to-component-norm} to see \eqref{es-eq1}  holds true.
The other two inequalities are essentially the same as \cite[Lemma 4.15]{GiHa2015combi}, as they are expressed component-wise.
\end{proof}

In the sequel, we will apply this lemma to $S$, the optimal recovery function of $M$ from Theorem \ref{th:S}. It is worth mentioning that when using Wendland functions $\Phi(\cdot):=\psi_{l,k}(c \|\cdot\|_2)$ with $ l:=\lfloor \frac{n}{2}\rfloor+k+1$ as reproducing kernels, $S$ is a linear combination of these functions and their first derivatives, see \eqref{S-form}; hence, $S\in C^{2k-1}(\R^n;\R^{n \times n})$.\\
Therefore, in order to be able to apply the lemma, we only consider Wendland functions with $k \geq 2$ (for more details, see for example \cite[section 3.2]{rbf2007giesl} or \cite[chapter 10]{wendland2005scattered}).

A CPA interpolation of $S$, or more exactly the values $P(x_k)=S(x_k)$ for all $x_k\in \cV_\cT$ for some triangulation $\cT$, that satisfies the constraints of the following semi-definite feasibility problem, is necessarily a contraction metric. 
Later we prove a converse statement: if $S$ is a contraction metric and $d\ge 2$ is fixed, then for any $h>0$ small enough  its CPA interpolation on an $(h,d)-$bounded triangulation will satisfy the constraints of the verification problem. Such triangulations are easily generated, see Remark \ref{trfullrem}.

\begin{VP} \label{VP}
Given is a system $\dot x=f(x)$, $f\in C^3(\R^n;\R^n)$, and a triangulation $\cT$ in $\R^n$. The verification problem has the following constants, variables, and constraints.

{\bf Constants:}
The constants used in the problem are%
\begin{enumerate}
\item $\epsilon_0>0$ -- lower bound on the matrix $P(x_k)$.
\item The diameter $h_\nu$ of each simplex $\T_\nu\in \cT${\rm :}
\begin{equation*}
  h_\nu :=\diam (\T_\nu)=\max_{x,y\in \T_\nu}\|x-y\|_2.
\end{equation*}

\item Upper bounds $B_\nu$ on the second-order derivatives of the components $f_k$ of $f$ on each simplex $\T_\nu\in\cT${\rm :}
\begin{equation*}
B_{\nu} \ge \max_{x\in \T_\nu \atop i,j,k=1, 2, \ldots, n }\left|\frac{\partial^2 f_k}{\partial x_i\partial x_j}(x)\right|.\label{B}
\end{equation*}
\item Upper bounds $B_{3,\nu}$ on the third-order derivatives of the components $f_k$ of $f$ on each simplex $\T_\nu\in\cT${\rm :}
\begin{equation*}
B_{3,\nu} \ge  \max_{x\in \T_\nu \atop i,j,k,l=1,\ldots,n }\left|\frac{\partial^3 f_l}{\partial x_i\partial x_j\partial x_k}(x)\right|.\label{B3}
\end{equation*}
\end{enumerate}

{\bf Variables:}
The variables of the problem are%
\begin{enumerate}
\item $P_{ij}(x_k) \in \mathbb R$ for all $1\le i\le j\le n$ and all vertices $x_k\in \cV_\cT$. 
 For $1\le i\le j\le n$ the value $P_{ij}(x_k)$ is the $(i,j)$-th entry of the $(n\times n)$ matrix $P(x_k)$. The matrix $P(x_k)$ is assumed to be symmetric and therefore these components determine it.%
\item $C_\nu \in \mathbb R_0^+$ for all simplices $\T_\nu\in \cT$ -- upper bound on $P$ in $\T_\nu$.
\item $D_\nu \in \mathbb R_0^+$ for all simplices $\T_\nu\in \cT$ -- upper bound on the derivative of $P_{ij}$ in $\T_\nu$.
\end{enumerate}

{\bf Constraints:}

\begin{enumerate}
\item {\bf Positive definiteness of $\mathbf P$}

For each $x_k\in \cV_\cT$\,{\rm :}%
\begin{equation*}
  P(x_k)\succeq  \epsilon_0 I.%
\end{equation*}

\item {\bf Upper bound on $\mathbf P$}

For each $x_k\in \cV_\cT$\,{\rm :}%
\begin{equation*}
  P(x_k)\preceq C_\nu I.%
\end{equation*}

\item {\bf Bound on the derivative of $\mathbf P$}

For each simplex $\T_\nu\in\cT$ and all $1\le i\le j\le n$\,{\rm :}%
\begin{equation*}
  \|w^{\nu}_{ij} \|_1 \le D_{\nu}.%
\end{equation*}
Here $w^{\nu}_{ij} = \nabla P_{ij}\big|_{\T_\nu^\circ}$ for all $x\in \T_\nu$, see Remark \ref{nugradrem} for details.%

\item {\bf Negative definiteness of $\mathbf A_{\nu}$}

For each simplex $\T_\nu = \co(x_0,\ldots,x_n)\in\cT$ and each vertex $x_k$ of $\T_\nu$\,{\rm :}%
\begin{equation*}
 -\epsilon_0 I \succeq A_{\nu}(x_k) + h_\nu^2  E_{\nu}I.
\end{equation*}
Here%
\begin{equation}\label{A-def}
  A_{\nu}(x_k) := P(x_k) D f(x_k) + D f(x_k)^T P(x_k) + (w_{ij}^\nu \cdot f(x_k))_{i,j=1, 2, \ldots, n},%
\end{equation}
where $D f(x_k)$ is the Jacobian matrix of $f$ at $x_k$, $(w_{ij}^\nu \cdot f(x_k))_{i,j=1, 2, \ldots, n}$ denotes the symmetric $(n\times n)$-matrix with entries $w_{ij}^\nu \cdot f(x_k)$ and $w^{\nu}_{ij}$ is defined as in \eqref{defw}, and for a fixed $\T_\nu$  and $i,j$ it is a constant vector independent of the vertex $x_k$ of $\T_\nu$. Further, %
\begin{equation*}
  E_{\nu} := n^{2} (1+4\sqrt{n}) B_\nu D_\nu+2 \, n^{3} B_{3,\nu}C_\nu.%
\end{equation*}
\end{enumerate}
\end{VP}

\begin{remark}\label{nugradrem}
In \textit{Constraints 3} and \textit{4} above, the gradient $w^\nu_{ij}$ of the affine function $P_{ij}\big|_{\T_\nu}$ on the simplex $\T_\nu=\co(x_0,\ldots,x_n)$, i.e.\ $\nabla P_{ij}\big|_{\T_\nu^\circ} = w_{ij}^\nu$, is given by the expression%
\begin{equation}\label{defw}
  w^{\nu}_{ij} := X^{-1}_{\nu}\left(\begin{array}{c}P_{ij}(x_1) - P_{ij}(x_0)\\ \vdots\\ P_{ij}(x_n) - P_{ij}(x_0)\end{array}\right)\in\mathbb R^n,%
\end{equation}
where $X_{\nu}=\left(x_1-x_0,x_2-x_0,\ldots,x_n-x_0\right)^T\in\mathbb R^{n\times n}$ is the so-called shape-matrix of the simplex $\T_\nu$. 

The \textit{Constraints 3} are indeed linear and can be implemented using the auxiliary variables $D_\nu^{k}$ and the constraints%
\begin{equation*}
  -D_\nu^{k}\le [w^\nu_{ij}]_k \le D_\nu^{k}\ \ \text{for $k=1, \dots, n$},%
\end{equation*}
where $[w^\nu_{ij}]_k $ is the $k$-th component of the vector $w^\nu_{ij}$, and setting $D_\nu=\sum_{k=1}^n D_\nu^{k}$.
\end{remark}

A feasible solution to Verification Problem \ref{VP}  delivers a symmetric matrix $P(x_k) = \left(P_{ij}(x_k)\right)_{i,j=1, 2, \ldots, n}$ at each vertex $x_k$ of the triangulation $\cT$ and values $C_\nu$ and $D_\nu$ for each simplex $\T_\nu\in\cT$.

We recall a lemma before expressing our final results.

%
%

\begin{lemma}[Operator estimate over a triangulation]\label{4.12} \cite[Lemma 4.9]{HaKa2019datalimit}

Assume $P$ is defined as in Definition \ref{CPA-def} from a feasible solution $P_{ij}(x_k)$ to the Verification Problem \ref{VP}. Fix a point $x\in\cD^\circ_\cT$ and a corresponding simplex $\T_\nu=\co(x_0,x_1,\ldots,x_n)\in \cT$ as in Definition \ref{basic}. Set %
\begin{equation*}
  A_{\nu}(y) := P(y) D f(y) + D f(y)^T P(y) + \left( w_{ij}^\nu \cdot f(y) \right)_{i,j=1, 2, \ldots , n}%
\end{equation*}
for all $y\in \T_\nu$. Then we have the following estimate with $x=\sum_{k=0}^n \lambda_k x_k$, $\lambda_k\ge 0$ and $\sum_{k=0}^n\lambda_k=1$ {\rm :}%
\begin{equation}\label{Aest}
  \left\|A_{\nu}(x)-\sum_{k=0}^n \lambda_k A_{\nu}(x_k)\right\|_2 \le h_\nu^2 E_\nu,
\end{equation}
in particular
\begin{equation*}
A_{\nu}(x) \preceq \, \sum_{k=0}^n \lambda_k A_{\nu}(x_k)+ h_\nu^2 E_\nu I.
\end{equation*}
\end{lemma}



We define the CPA metric $P$ by affine interpolation on each simplex. The following theorem explains why we call Problem \ref{VP} a Verification Problem as it shows that if the finitely many constraints at vertices are satisfied, then the interpolated CPA function is a Riemannian contraction metric on $\cD_\cT^\circ$.

\begin{theorem}[CPA contraction metric]\label{le1} \hspace*{1cm} \\
	Let $f\in C^3(\mathbb R^n,\mathbb R^n)$.
Assume the constraints of  Verification Problem \ref{VP} are satisfied for some values $P_{ij}(x_k)$, $C_\nu$, $D_\nu$.  Then the matrix-valued function $P$, where $P(x)$ is interpolated from the values $P_{ij}(x_k)$ as in  Definition \ref{CPA-def}, is a Riemannian metric contracting in any compact set $K\subset \cD_\cT^\circ$.
\end{theorem}

\begin{proof}
Let $ x \in \cD_\cT$ be an arbitrary point, $x=\sum_{k=0}^n \lambda_k x_k$, $\lambda_k\ge 0$ and $\sum_{k=0}^n\lambda_k=1$, with a corresponding $ \T_{\nu} \in \cT $. The symmetry of $P(x)$ follows directly from  $P_{ij}(x_k)=P_{ji}(x_k)$ assumed in \textit{Variables 1} of Verification Problem \ref{VP}:
\begin{equation*}
P_{ij}(x) = P_{ij} \left(\sum_{k=0}^n \lambda_k x_k \right)
			 =  \sum_{k=0}^n \lambda_k P_{ij}(x_k)
			 = \sum_{k=0}^n \lambda_k P_{ji}(x_k)
			 = P_{ji}(x).
\end{equation*}
For positive definiteness, we have $ P(x_k) \succeq  \epsilon_0 I $ for each $x_k\in \cV_\cT$
by \textit{Constraints 1}, so
\begin{equation*}
P(x) =  \sum_{k=0}^n \lambda_k P(x_k)  \succeq \sum_{k=0}^n \lambda_k \, \epsilon_0 I = \epsilon_0 I.
\end{equation*}

Now let $x\in K\subset \cD_\cT^\circ$. Then there is a simplex  $ \T_{\nu} \in \cT $ with $x\in \T_\nu$ as well as $x+\theta f(x) \in \T_\nu$ for all $\theta \in [0,\theta^*]$ with $\theta^*>0$. Then, as
expressed in Remark \ref{basic}, we can show that
$ w^{\nu}_{ij} = \nabla P_{ij}\big|_{\T_\nu^\circ}(x)$ and  $(P_+')_{ij}(x)=w^\nu_{ij}\cdot f(x)$.
Hence,
\begin{eqnarray*}
	{\cal L}_P(x)&=&\max_{v^TP(x)v=1}L_P(x;v)\\
	&=&\max_{v^TP(x)v=1}\frac{1}{2}v^T[P(x)Df(x)+Df(x)^TP(x)+
	P_+'(x)]v\\
		&=&\max_{v^TP(x)v=1}\frac{1}{2} v^TA_\nu(x)v\\
	&\le&\max_{v^TP(x)v=1}\sum_{k=0}^n\lambda_kv^T
	[A_\nu(x_k)+h_\nu^2E_\nu I] v\\
	&\le&-\epsilon_0\max_{\sum_{k=0}^n\lambda_k v^TP(x_k)v=1}\sum_{k=0}^n\lambda_k\|v\|^2_2\\
	&=& -\frac{\epsilon_0}{C_\nu}\max_{\sum_{k=0}^n\lambda_k v^TP(x_k)v=1}\sum_{k=0}^n\lambda_k v^TP(x_k)v\\
	&=&-\frac{\epsilon_0}{C_\nu}\ <\ 0
	\end{eqnarray*}
in which we have used Lemma \ref{4.12}, \textit{Constraints 2}, and \textit{Constraints 4}.
%
\end{proof}

In order to measure how good the CPA interpolant $P$ of the RBF approximation $S$ of the contraction metric $M$ is, we need to check two criteria in correspondence to two properties of the contraction metric. First, how close $P$ is to $M$, and second, how close $F(P)$ is to $F(M)$. The following lemma provides these estimates.

\begin{lemma} [{\small Error estimate for RBF-CPA approximation of the contraction metric}] \label{RBF-CPA-estimate}
Let $k\ge 2$ if $n$ is odd and $k\ge 3$ if $n$ is even. Assume that $x_0$ is an exponentially stable equilibrium of $\dot x =f(x)$ where
$f\in C^{\sigma + 1}(\R^n;\R^n)$, with $ \sigma \in \N $ and $\sigma \geq k + \frac{n+1}{2}$.
Let $C\in\Sb^{n\times n}$ be a positive definite matrix and $M\in C^{\sigma}(\cA(x_0);\Sb^{n\times n})$ be
the solution of the PDE (\ref{matrixeq4}) from Theorem \ref{un}.
Let $K \subset \Omega \subset \cA(x_0)$ be a positively invariant and compact set, where $\Omega$ is open with Lipschitz boundary; and let
$S \in H^\sigma(\Omega;\Sb^{n\times n}) $ be the optimal recovery of $M$ from Theorem \ref{th:S} with kernel given by the Wendland function
$\psi_{l,k}$ with $l=\lfloor \frac{n}{2}\rfloor +k+1$ and collocation points $X\subset \Omega$. Finally, let
$P$ be the CPA interpolation of $ S $ on an $ (h, d) $-bounded triangulation $ \cT = \{\T_\nu\} $ with
$  K \subset \cD_{\cT}^\circ \subset \cD_{\cT} \subset \Omega   $ that satisfies the constraints of Verification Problem \ref{VP}.
Then, we have for all small enough $h_{X,\Omega}>0$ the following error estimates\,{\rm:}
\begin{eqnarray}
\sup_{x\in K} \norm{M(x)-P(x)}{2} & \le & \Big( \beta \,  h_{X,\Omega}^{\sigma-1-n/2} + \zeta \, n \, h^2 \,  \Big) \|M\|_{H^\sigma(\Omega;\Sb^{n\times n})}
,\label{estimateM}\\
\sup_{x\in K} \norm{F(M-P)(x)}{2} & \le & \left( \frac{\beta}{\alpha} \,  h_{X,\Omega}^{\sigma-1-n/2} \,
+ \eta \, h \,  \norm{f}{C^{1}(\Omega;\R^{n})} \right)
 \|M\|_{H^\sigma(\Omega;\Sb^{n\times n})},~~~~~~~ \label{estimateF}
\end{eqnarray}
where $\zeta$ is the constant from Remark \ref{norm-remark}, $ \gamma =  1 + \frac{d n^{3/2}}{2} $, and $ \eta =n \, \zeta \, \left(2  h + \gamma \right) $.
\end{lemma}

\begin{proof}
First, note that by Theorem \ref{thmfinal} we have
\begin{eqnarray*}
\sup_{x\in K} \norm{M(x)-P(x)}{2} & \le & \sup_{x\in K} \big( \norm{M(x)-S(x)}{2}+ \norm{S(x)-P(x)}{2} \big) \\
& \leq & \beta h_{X,\Omega}^{\sigma-1-n/2} \|M\|_{H^\sigma(\Omega;\Sb^{n\times n})} + \sup_{x\in K} \norm{S(x)-P(x)}{2} .
\end{eqnarray*}
Next we provide an estimate for the latter term over each simplex $\fS_\nu$, using Lemma \ref{CPA-estimate}, Remark \ref{norm-remark}, and noting that $ K \subset \cD_{\cT} \subset \Omega $ allows us to introduce an open set $\cD$ with $C^1 $ boundary such that $\cD_{\cT} \subset \cD \subset \Omega $:
\begin{eqnarray*}
\sup_{x\in K} \norm{S(x)-P(x)}{2} & \leq & \sup_{\nu} \sup_{x \in \fS_\nu } \|S(x)-P(x)\|_{2} \\
& \leq & n \, h^2 \, \|S\|_{C^{2}(\cD;\Sb^{n\times n})}\\
& \leq &  \zeta \,n \, h^2 \, \|M\|_{H^\sigma(\Omega;\Sb^{n\times n})}.
\end{eqnarray*}
This shows the first estimate. \\
The same procedure can be used for the second claimed estimate:
\begin{eqnarray*}
\sup_{x\in K} \norm{F(M-P)(x)}{2} & \le & \sup_{x\in K} \big( \norm{F(M)(x)-F(S)(x)}{2}  + \norm{F(S)(x)-F(P)(x)}{2} \big)  \\
& \le & \frac{\beta}{\alpha} h_{X,\Omega}^{\sigma-1-n/2} \|M\|_{H^\sigma(\Omega;\Sb^{n\times n})} + \sup_{x\in K} \norm{F(S-P)(x)}{2}.
\end{eqnarray*}
Now, for the second term, let $x\in K\subset \cD_{\cT} $, so there exists a simplex such that $x\in  \fS_\nu$. Then by using Remark \ref{basic} we have
\begin{small}
\begin{equation*}
F(S-P)(x) = \big( S(x)-P(x) \big) Df(x) + Df(x)^T \big( S(x)-P(x) \big) +
\Big( \big[\nabla S_{ij}(x) - w_{ij}^{\nu}\big] \cdot f(x) \Big)_{i,j=1, 2, \ldots , n}.
\end{equation*}
\end{small}
Observe that $ \norm{Df(x)^T}{2} =  \norm{Df(x)}{2} $, and by inequality \eqref{D-C1-norm-Rn} of Remark \ref{function-norm-inequalities} we get
\begin{equation*}
 \sup_{\nu} \sup_{x \in \fS_\nu } \norm{S(x)-P(x)}{2} \norm{Df(x)}{2}  \leq
n \, h^2 \, \zeta \, \norm{M}{H^\sigma(\Omega;\Sb^{n\times n})} \norm{f}{C^{1}(\Omega;\R^{n})}.
\end{equation*}
Let $Q_\nu(x) = \Big(  \nabla S_{ij}(x)  \cdot f(x) - w_{ij}^\nu \cdot f(x) \Big)_{i,j=1, 2, \ldots , n} $. From the H\"older inequality and inequality \eqref{es-eq2} of Lemma \ref{CPA-estimate} we have the following estimate
\begin{eqnarray}\label{Q-estimate}
\sup_{\nu} \sup_{x \in \fS_\nu } \norm{Q_\nu(x)}{2} & \leq & n \sup_{\nu} \sup_{x \in \fS_\nu } \norm{Q_\nu(x)}{\max} \nonumber \\
&  \leq &  n \sup_{\nu} \sup_{x \in \fS_\nu } \max_{i, j = 1, 2, \cdots, n}
\norm{ \nabla S_{ij}(x)  - w_{ij}^\nu}{1} \norm{f(x)}{\infty} \nonumber \\
&\leq&  n \, h \, \gamma  \, \norm{S}{C^2(\cD;\Sb^{n \times n})}\norm{f}{C^{0}(\Omega;\R^{n})} \\
&\leq&  n \, h \, \gamma \, \zeta \,
\norm{M}{H^\sigma(\Omega;\Sb^{n\times n})} \norm{f}{C^{1}(\Omega;\R^{n})}. \nonumber
\end{eqnarray}
Putting all terms together delivers
\begin{eqnarray*}
\sup_{x \in K} \norm{F(S)(x)-F(P)(x)}{2} & \leq & \sup_{\nu} \sup_{x \in \fS_\nu } \norm{F(S)(x)-F(P)(x)}{2} \\
& \leq & n \, h \, \zeta \, \left( 2  h + \gamma \right) \norm{f}{C^{1}(\Omega;\R^{n})} \norm{M}{H^\sigma(\Omega;\Sb^{n\times n})}.
\end{eqnarray*}
It is then just a simplification of coefficients to get \eqref{estimateF} and the proof is complete.
\end{proof}

The last theorem of this section proves that a suitable CPA interpolation of the solution  to a suitable  optimal recovery problem,   will indeed be a contraction metric.


\begin{remark}\label{trfullrem}
The following observation is useful for the statement of the next theorem: Given an open set $\cD$, compact set $ \widetilde{K} \subset \cD$, and $d=2\sqrt{n}$, one can always construct an $(h,d)$-bounded triangulation $\cT$ such that $\widetilde{K} \subset D_\cT^\circ\subset D_\cT\subset \cD $.  Indeed, by \cite[Lemma 4.9]{GiHa2015combi} the so-called scaled standard triangulation
$\cT^{\rm std}_\rho$ is $(h,4n)$-bounded for any $h>\sqrt{n}\,\rho$.
In \cite[Remark 2]{hafstein2017study} a sharper bound is derived, which shows that $\cT^{\rm std}_\rho$ is even
$(h,2\sqrt{n})$-bounded.  By setting $3\epsilon :=\dist(\widetilde{K},\R^n\setminus \cD)=\min\{\|x-y\|\,:\, x\in \widetilde{K},y\in \R^n\setminus \cD\}$ and
$K_\epsilon := \{x\in\R^n\,:\, \dist(x,\widetilde{K})< \epsilon\}$, it is easy to see that with $0<\rho \le \epsilon/\sqrt{n}$ the triangulation
$
\T:=\{\T_\nu\in \cT^{\rm std}_\rho\,:\, \T_\nu \cap K_\epsilon \not=\emptyset\}
$
fulfills $ \widetilde{K} \subset D_\cT^\circ\subset D_\cT\subset \cD $.
\end{remark}

\begin{theorem}[RBF-CPA contraction metric] \label{RBF-CPA contraction metric} 
Let $k\ge 2$ if $n$ is odd and $k\ge 3$ if $n$ is even. Define $\sigma=k+\frac{n+1}{2}$ and assume that $x_0$ is an exponentially stable equilibrium of $\dot x =f(x)$ where
 $f\in C^{\sigma+1}(\mathbb R^n;\mathbb R^n)$.
Let $C\in\Sb^{n\times n}$ be a positive definite matrix and $M\in C^{\sigma}(\cA(x_0);\Sb^{n\times n})$ be
the solution of the PDE (\ref{matrixeq4}) from Theorem \ref{un}, i.e.~PDE (\ref{matrixeq2}) with a constant right-hand-side.

Let $ \Omega \subset \cA(x_0)$ be open and bounded with Lipschitz boundary and
let ${\cal D} \subset\Omega$ be positively invariant, open set with $C^1$ boundary, such that $x_0\in \cal D $ and $\overline{\cD} \subset \Omega$.

Fix a compact set $ \widetilde{K} \subset \cD$ and constants
$$
d\ge 2\sqrt{n},\ \  B^* \ge \max_{x\in \overline{\cD}  \atop i,j,k=1, 2, \ldots, n }\left|\frac{\partial^2 f_k}{\partial x_i\partial x_j}(x)\right|,\ \ \text{and}\ \
B_{3}^* \ge  \max_{x\in \overline{\cD} \atop i,j,k,l=1,\ldots,n }\left|\frac{\partial^3 f_l}{\partial x_i\partial x_j\partial x_k}(x)\right|.
$$

Then there exist constants $h^*_{X,\Omega}, h^* > 0 $, such that for any set of collocation points $X\subset \Omega$ with fill distance $  h_{X,\Omega} \leq h^*_{X,\Omega} $ and any $(h,d)$-bounded triangulation $\cT$ with $ \widetilde{K}\subset D_\cT^\circ\subset D_\cT\subset \cD $ and $ h < h^* $ the following holds:
Suppose that $S$ is the optimal recovery of $M$ from Theorem \ref{th:S} with kernel given by the Wendland function
$\psi_{l,k}$ with $l=\lfloor \frac{n}{2}\rfloor +k+1$. Fix the constants and variables from Verification Problem \ref{VP} as follows for all $\T_\nu \in \cT$, $x_k\in \cV_\cT$, and $1\le i \le j \le n$\,{\rm :}
$$
P_{ij}(x_k)=S_{ij}(x_k), \ \ C_\nu= \max\{\|P(x)\|_2\,:\,x\ \text{vertex of $\T_\nu$} \}, \ \
D_\nu =\norm{w_{ij}^\nu}{1},
$$
$$
B^* \ge B_{\nu} \ge \max_{x\in \T_\nu \atop i,j,k=1, 2, \ldots, n }\left|\frac{\partial^2 f_k}{\partial x_i\partial x_j}(x)\right|,\ \
B_{3}^* \ge B_{3,\nu} \ge  \max_{x\in \T_\nu \atop i,j,k,l=1,\ldots,n }\left|\frac{\partial^3 f_l}{\partial x_i\partial x_j\partial x_k}(x)\right|,
$$
and
$$
\epsilon_0 = \frac{1}{3} \min(\lambda_0,\lambda_1) >0, \ \ \text{where}\ \ \lambda _0 I \preceq  M(x)\ \ \text{for all $x\in \overline{\Omega}$} \ \ \text{and} \ \ \lambda_1 I \preceq C.
$$
Then the constraints of Verification Problem \ref{VP} are fulfilled by these values.

In particular, we can assert that the CPA interpolation $P$ of $S$ on $\cT$ is a contraction metric on $\widetilde{K}$.
%
\end{theorem}

\begin{proof}
First note that since $\overline{\Omega}$ is compact and $M$ is positive definite by Theorem \ref{un}, there are constants $\lambda_0,\lambda_1,\Lambda_0>0$ such that for all $x\in \overline{\Omega}$ we have
	\begin{eqnarray}
		\lambda _0 I &\preceq&  M(x)\  \preceq\  \Lambda_0 I \label{ab1}\\
		\lambda_1I& \preceq&C.\label{ab2}
		\end{eqnarray}
		Furthermore, define
		\begin{eqnarray*}
			C^*&:=&\Lambda_0+\frac{2}{3}\lambda_0,\\
			D^*&:=&(1+\gamma)\zeta \|M\|_{H^\sigma(\Omega;\Sb^{n\times n})},\\
			E^*&:=&n^2(1+4\sqrt{n})B^*D^*+2n^3B_3^*C^*,
			\end{eqnarray*}
where $\gamma=1+\frac{dn^{\frac32}}{2}$ is the constant from Lemma \ref{CPA-estimate}, $\alpha,\beta>0$ are the constants from Theorem \ref{thmfinal} with $K=\overline{\cD}$, and $\zeta>0$ is the constant from Remark \ref{norm-remark}.

	Now set
	\begin{eqnarray*}
		h^*&:=&\min\left(\sqrt{\frac{\lambda_0}{3\zeta n \|M\|_{H^\sigma(\Omega;\Sb^{n\times n})}}}
	,\frac{\lambda_1}{3\left( n\gamma \zeta  \|f\|_{C^0(\Omega;\mathbb R^n)}  \|M\|_{H^\sigma(\Omega;\Sb^{n\times n})}+E^*\right)},1\right),\\
		h^*_{X,\Omega}&:=&\min\left(\frac{\lambda_0}{3\beta \|M\|_{H^\sigma(\Omega;\Sb^{n\times n})}},\frac{\lambda_1\alpha}{3\beta \|M\|_{H^\sigma(\Omega;\Sb^{n\times n})}}\right)^{1/(\sigma-1-n/2)}.
		\end{eqnarray*}
		
Note that the $C_\nu$s and $D_\nu$s are defined as the minimal number such that \textit{Constraints 2} and \textit{Constraints 3} of Verification Problem \ref{VP} are satisfied.
Now we verify that \textit{Constraints 1} are fulfilled.  
		
First note that by the construction method of Theorem \ref{th:S}, we know that the $S(x_k)$ and hence the $P(x_k) $ are  symmetric matrices.

We have for all $x\in  D_\cT \subset \cD \subset \overline{\cD}\subset \Omega$ that
\begin{eqnarray*}
	P(x)&=&M(x)-M(x)+P(x)\\		
		&\succeq&\left(\lambda_0-\Big( \beta \,  h_{X,\Omega}^{\sigma-1-n/2} + \zeta \, n \, h^2 \,  \Big) \|M\|_{H^\sigma(\Omega;\Sb^{n\times n})}	
		\right)I
			\succeq\frac{\lambda_0}{3}I \succeq\epsilon_0 I,
			\end{eqnarray*}
where we used inequalities \eqref{ab1}, \eqref{estimateM} with $K=\overline{\cD}$, and the definitions of $h^*,h^*_{X,\Omega}$, and $\epsilon_0$.
Thus, \textit{Constraints 1} hold true.

We now show that $C_\nu\le C^*$.
			
			We have for all $x\in  D_\cT \subset \cD \subset \overline{\cD}\subset \Omega$, similarly to above, that
			\begin{eqnarray*}
				P(x)&=&M(x)-M(x)+P(x)\\	
				&\preceq&\left(\Lambda_0+ \Big( \beta \,  h_{X,\Omega}^{\sigma-1-n/2} + \zeta \, n \, h^2 \,  \Big) \|M\|_{H^\sigma(\Omega;\Sb^{n\times n})}		
				\right)I\\	
	& \preceq  & \left(\Lambda_0+\frac{2\lambda_0}{3}\right)I \\
	&  \preceq & C^*.
\end{eqnarray*}
Since $C_\nu$ were chosen as the smallest constants to satisfy \textit{Constraints 2}, we must have $0<C_\nu\le C^*$.

%

We show that $D_\nu\le D^*$.
Consider a simplex $\T_\nu\in\cT$ and let $1\le i\le j\le n$.
\begin{eqnarray*}	\|w^{\nu}_{ij} \|_1 &=&
\|	\nabla P_{ij}\big|_{\T_\nu^\circ}\|_1\\
&\le&(1+h\gamma)\|S\|_{C^2(\overline{\cD};\Sb^{n\times n})}\\
&\le&(1+\gamma)\zeta\|M\|_{H^\sigma(\Omega;\Sb^{n\times n})}\\
	& \leq & D^*,
\end{eqnarray*}
where we used inequalities \eqref{es-eq3}, \eqref{norm-remark-eq}, $h\le h^*\le 1$, and the definition of $D^*$.
Since $D_\nu$ were chosen as the smallest constants to satisfy \textit{Constraints 3}, we have $0<D_\nu\le D^*$.

To show that \textit{Constraints 4} are fulfilled, it is advantageous to first derive the following upper bound on $E_{\nu}$;
\begin{eqnarray*}
E_{\nu} & = & n^{2} (1+4 \sqrt{n}) B_\nu D_\nu+2 \, n^{3} B_{3,\nu}C_\nu \\
& \leq & n^{2} (1+4 \sqrt{n}) B^* D^*+ 2 n^3  B^*_3 C^* \\
&=&E^*.
\end{eqnarray*}
%

To conclude the proof fix a simplex  $\T_\nu\in\cT$ and let $x_k$ be one of its vertices. Then $x_k\in  D_\cT \subset \cD \subset \overline{\cD}\subset \Omega$. Since $P(x_k)=S(x_k)$ we get by \eqref{es-eq2}, and \eqref{Q-estimate}
\begin{eqnarray*}
	A_\nu(x_k) &=&
	P(x_k)Df(x_k)+Df(x_k)^TP(x_k)+(w^\nu_{ij}\cdot f(x_k))_{i,j=1,2,\ldots,n}\\
	 &=&
	S(x_k)Df(x_k)+Df(x_k)^TS(x_k)+
	(\nabla S_{ij}(x_k)\cdot f(x_k))_{i,j=1,2,\ldots,n}\\
	&&
	+((w^\nu_{ij}-\nabla S_{ij}(x_k))\cdot f(x_k))_{i,j=1,2,\ldots,n}\\
	  &\preceq  & {F(S)(x_k)+n \cdot\max_{i,j=1,\ldots,n}\|w^\nu_{ij}-\nabla S_{ij}(x_k)\|_1\ \sup_{x\in {\Omega}} \|f(x)\|_\infty I} \\
	   &\preceq  &
	F(M)(x_k)+F(S)(x_k)-F(M)(x_k)+nh\gamma\|S\|_{C^2(\overline{\cD};\Sb^{n\times n})}\|f\|_{C^0(\Omega;\mathbb R^n)}I\\
	&\preceq  &
	-C+ \left(\frac{\beta}{\alpha} \,  h_{X,\Omega}^{\sigma-1-n/2} +nh\gamma\zeta \|f\|_{C^0(\Omega;\mathbb R^n)}\right)\,\|M\|_{H^\sigma(\Omega;\Sb^{n\times n})}I,
\end{eqnarray*}
where the last inequality follows by \eqref{RBF-est} and \eqref{norm-remark-eq}.

Hence,
\begin{eqnarray*}
	\lefteqn{
A_\nu(x_k) + h_\nu^2 E_\nu I}\\
  &\preceq  &	\left(-\lambda_1+ \left(\frac{\beta}{\alpha} \,  h_{X,\Omega}^{\sigma-1-n/2} +nh\gamma\zeta
\|f\|_{C^0(\Omega;\R^n)}\right)\,\|M\|_{H^\sigma(\Omega;\Sb^{n\times n})}+(h^*)^2 E^*\right)
 I\\
&\preceq  &
\left(-\lambda_1+ \frac{2}{3}\lambda_1\right) I \preceq -\epsilon_0 I,
\end{eqnarray*}
by  \eqref{ab2} and the definitions of $h^* $, $h^*_{X,\Omega}$, and $\epsilon_0$.  This concludes our proof.
\end{proof}

\section{Examples} \label{examples-section}
When applying the method to examples, we choose the symmetric and positive definite matrix $C$ on the right-hand side of \eqref{matrixeq2} to be the identity matrix, $C=I$.
We will first review the steps of the method in Section \ref{method-subsection} and then provide two examples.

\subsection{The Method}\label{method-subsection}
Given is a system $\dot x=f(x)$, with $f\in C^{\sigma + 1}(\R^n;\R^n)$, where $ \sigma=k+\frac{n+1}{2}$ and $k\ge 2$ if $n$ is odd and $k\ge 3$ if $n$ is even, so that the minimum smoothness needed for the contraction metric $M$ and its optimal recovery $S$ is guaranteed (by Theorems \ref{thmfinal} and \ref{RBF-CPA contraction metric}).

The idea is to to increase the number of collocation points and vertices gradually so that we obtain a small enough fill distance and fine enough triangulation.
Theorem \ref{RBF-CPA contraction metric} ensures that after finitely many repetitions the conditions of Verification Problem \ref{VP} will be satisfied. In other words, this is a semi-decidable problem, i.e.~if there exist a contraction metric, then we can compute one in a
finite number of steps.   The steps of the method are as follows:

\begin{enumerate}
\item[STEP 0.] Fix $d \ge 2\sqrt{n}$, $h_\text{collo}>0$, $h_\text{triang}>0$, $c>0$, $k\ge 2$ if $n$ is odd and $k\ge 3$ if $n$ is even, and  the Wendland function $\psi_0(r):=\psi_{l,k}(cr)$ with $l=\lfloor\frac{n}{2}\rfloor+k+1$. 
     Denote  $\psi_{q+1}(r)=\frac{1}{r}\frac{d\psi_q}{dr}(r)$ for $q=0,1$.  Further, fix the compact set $\widetilde K\subset \R^n$, where we want to compute a
    contraction metric for the system, the upper bounds $B^*$ and $B^*_3$ as in Theorem \ref{RBF-CPA contraction metric}, and an open set $\Omega \supset \widetilde K$.

\item[STEP I.]
    Choose a set of pairwise distinct points $X=\{x_1,\ldots,x_N\} $ in $\Omega$ as collocation points with fill distance $h_{X,\Omega} \le h_\text{collo}$.
To obtain a solution of the optimal recovery problem based on RBF approximation we follow these steps:

\begin{enumerate}
\item[1.]
Compute the coefficients $ b_{k,\ell,i,j,\mu,\nu} $ with
\begin{eqnarray}\label{b-coeff}
b_{k,\ell,i,j,\mu,\nu}&=&
\psi_0(\|x_k-x_\ell\|_2)\bigg[\sum_{p=1}^n Df_{pi}(x_\ell)Df_{p\mu}(x_k)\delta_{\nu j}+Df_{\mu i}(x_\ell) Df_{j\nu}(x_k)\nonumber\\
&&
+Df_{i\mu}(x_k) Df_{\nu j}(x_\ell)+\delta_{i\mu} \sum_{p=1}^n Df_{p\nu}(x_k)Df_{pj}(x_\ell)\bigg]\nonumber\\
&&+\psi_1(\|x_k-x_\ell\|_2)\langle x_k-x_\ell,f(x_k)\rangle \left[Df_{\mu i}(x_\ell)\delta_{\nu j}+\delta_{i\mu} Df_{\nu j}(x_\ell)\right]\nonumber\\
&&+\psi_1(\|x_k-x_\ell\|_2)\langle x_\ell-x_k,f(x_\ell)\rangle \left[Df_{i\mu}(x_k)\delta_{\nu j}+\delta_{i\mu}Df_{j\nu}(x_k)\right]\nonumber\\
&&-\psi_1(\|x_k-x_\ell\|_2)\langle f(x_\ell),f(x_k)\rangle\delta_{i\mu}\delta_{j\nu}\nonumber\\
&&+\psi_2(\|x_k-x_\ell\|_2))\langle x_k-x_\ell,f(x_k)\rangle)\langle x_\ell-x_k,f(x_\ell)\rangle\delta_{i\mu}\delta_{j\nu}.
\end{eqnarray}
for $ 1 \leq k, \ell \leq N $, and $ 1 \leq i, j, \mu, \nu \leq n $ (see \cite[Subsection 3.2]{giesl-cont-1} for more details).

\item[2.] Calculate the coefficients $ c_{k,\ell,i,i,\mu,\mu} $ with
\begin{eqnarray}\label{c-form}
c_{k,\ell,i,i,\mu,\mu}&=&b_{k,\ell,i,i,\mu,\mu}\nonumber\\
c_{k,\ell,i,i,\mu,\nu}&=&\frac{1}{2}\left(b_{k,\ell,i,i,\mu,\nu}+b_{k,\ell,i,i,\nu,\mu}\right)\nonumber\\
c_{k,\ell,i,j,\mu,\mu}&=&b_{k,\ell,i,j,\mu,\mu}
=
\frac{1}{2}\left(b_{k,\ell,i,j,\mu,\mu}+b_{k,\ell,j,i,\mu,\mu}\right)\nonumber\\
c_{k,\ell,i,j,\mu,\nu}&=&\frac{1}{2}\left(b_{k,\ell,i,j,\mu,\nu}+b_{k,\ell,i,j,\nu,\mu}\right)\nonumber
\\
&=&\frac{1}{4}\left(b_{k,\ell,i,j,\mu,\nu}+b_{k,\ell,j,i,\nu,\mu}+b_{k,\ell,i,j,\nu,\mu}
+b_{k,\ell,j,i,\mu,\nu}
	\right)
\end{eqnarray}
where we assume $\mu<\nu$ and $i<j$.

\item[3.]
Determine $\gamma_k^{(\mu,\nu)}$, by solving the linear system
\begin{eqnarray}\label{LIN}
	\sum_{k=1}^N\sum_{1\le \mu\le \nu\le n} c_{k,\ell,i,j,\mu,\nu}\gamma_k^{(\mu,\nu)}
	=(F (S)(x_\ell))_{i,j}
		=\lambda_\ell^{(i,j)}(S)
=	-C_{ij}
\end{eqnarray}
for $ 1 \leq \ell \leq N $, and $ 1 \leq i \leq j \leq n $.  Note that \eqref{LIN} is a system of $Nn(n+1)/2$ equations in $Nn(n+1)/2$ unknowns.

\item[4.]
Compute $\beta_k\in \Sb^{n\times n}$ from $\gamma_k$; recalling that
\begin{eqnarray*}
\beta_k^{(j,i)}&=&\beta_k^{(i,j)}=\frac{1}{2}\gamma_k^{(i,j)}  \mbox{~~~~if~~} i\not= j, \\
\beta_k^{(i,i)}&=&\gamma_k^{(i,i)}.
\end{eqnarray*}

\item[5.]
We  now have a formula for the optimal recovery
\begin{eqnarray}\label{S-form}
S(x)  &= & \sum_{k=1}^N \bigg[ \psi_0(\|x_k-x\|_2)\left[ Df(x_k)\beta_k+ \beta_k Df (x_k)^T \right] \nonumber \\
  && ~~~~~ + \, \psi_1(\|x_k-x\|_2)\langle x_k-x,f(x_k)\rangle \beta_k\bigg].
\end{eqnarray}
\end{enumerate}

\item[STEP II.]  Fix an $(h,d)$-bounded triangulation $\cT$ with $h \le  h_\text{triang}$ and $\cD_\cT=\widetilde K$. 
Moreover, fix constants $B_\nu \le B^*$ and $B_{3,\nu} \le B^*_3$ satisfying \eqref{B} and \eqref{B3}. Compute the values $S(y)$ at the vertices of the triangulation $y\in\cV_\cT$ and check if they are positive definite. If not, decrease $h_\text{collo}$ by a factor, e.g.~$h_\text{collo} \leftarrow h_\text{collo} /2$, and go back to STEP I. 
\item[STEP III.]
Use the formulas in  Theorem \ref{RBF-CPA contraction metric} to compute the $C_\nu$s and $D_\nu$s and
check whether {\it Constraints 4} of Verification Problem \ref{VP} are fulfilled ({\it Constraints 2 and 3} are automatically fulfilled).
If not, then reduce $h_\text{triang}$ by a factor, e.g.~$h_\text{triang} \leftarrow h_\text{triang}/8$, and repeat STEP II.
If the conditions still don't hold, decrease $h_\text{collo}$ by a factor, e.g.~$h_\text{collo} \leftarrow h_\text{collo} /2$, and go back to STEP I.
\item[STEP IV.]
Build the contraction metric $P:\widetilde K \to \Sb^{n\times n}$ as the CPA interpolation of the values $P(y)$, $y\in\cV_\cT$,  as suggested in Definition \ref{CPA-def}.
\end{enumerate}

\begin{remark}
Note that in most applications it is more practical to use a relaxed version of the procedure above to compute a contraction metric.  If the matrices $P(y)$,  $y\in\cV_\cT$, in STEP II are positive definite in a reasonably large
part of $\widetilde K$, then one can proceed to STEP III.   Further, if additionally  {\it Constraints 4} of Verification Problem \ref{VP} are fulfilled in a reasonably large part of $\widetilde K$ in STEP III, then one can proceed to STEP IV.
The CPA interpolation $P$ will then not be a contraction metric on the whole of $\widetilde K$, but on the subset where it is both positive definite and fulfills {\it Constraints 4} of Verification Problem \ref{VP}.  We use this
relaxed procedure in the examples below.
\end{remark}

\subsection{Van der Pol System}
As an example, we consider the classical van der Pol equation with reversed time
\begin{eqnarray}
\left\{\begin{array}{lcl}	\dot{x}&=&-y\\
	\dot{y}&=&x-3(1-x^2)y\end{array}\right.\label{pol}
\end{eqnarray}
and denote the right-hand side by $f(x,y)$.
It is well known to have an exponentially stable equilibrium at the origin with basin of
attraction bounded by an unstable periodic orbit.  We demonstrate  the applicability of our method to this well known example.

The kernel given by  Wendland's function
$$ \psi_{6,4}(r)=(1-cr)_+^{10} \big(2145 (cr)^4+2250 (c r)^3+1050 (cr)^2 + 250cr+25 \big)$$
with $c=0.9$ is used with corresponding RKHS $H^\sigma$ with $\sigma=4+\frac{2+1}{2}=5.5$. 
We used $ N = 1926 $ collocation points and a hexagonal grid \cite{iske2} to cover the area inside the periodic orbit. Then, we calculated the CPA verification over the rectangle
$[-2.5,~2.5] \times [-5.5,~5.5] $ with $2200^2 $ vertices, see Figure \ref{fig1_pol}.

\begin{figure}[ht]
\centering	
	\includegraphics[height =5cm, width=0.5 \textwidth]{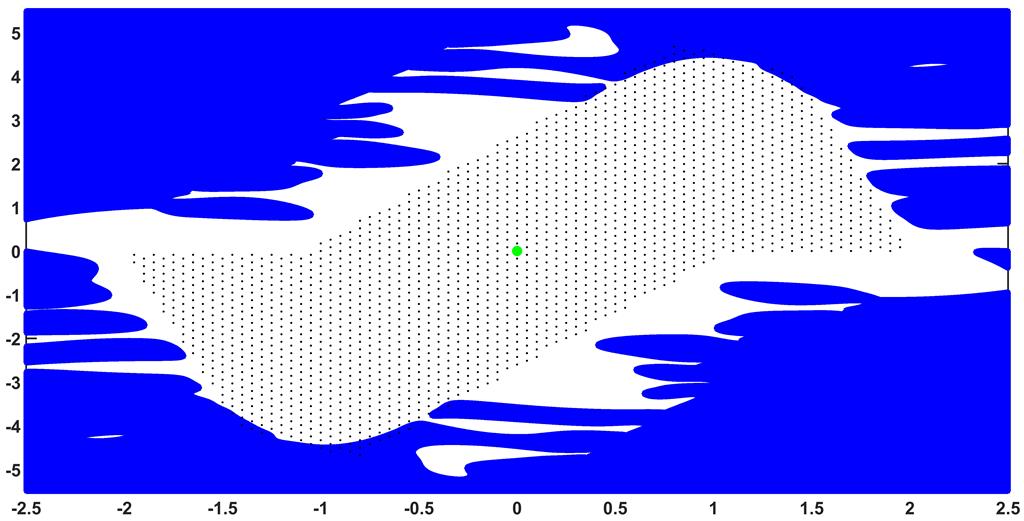}~%
	\includegraphics[height=5cm, width=0.5 \textwidth]{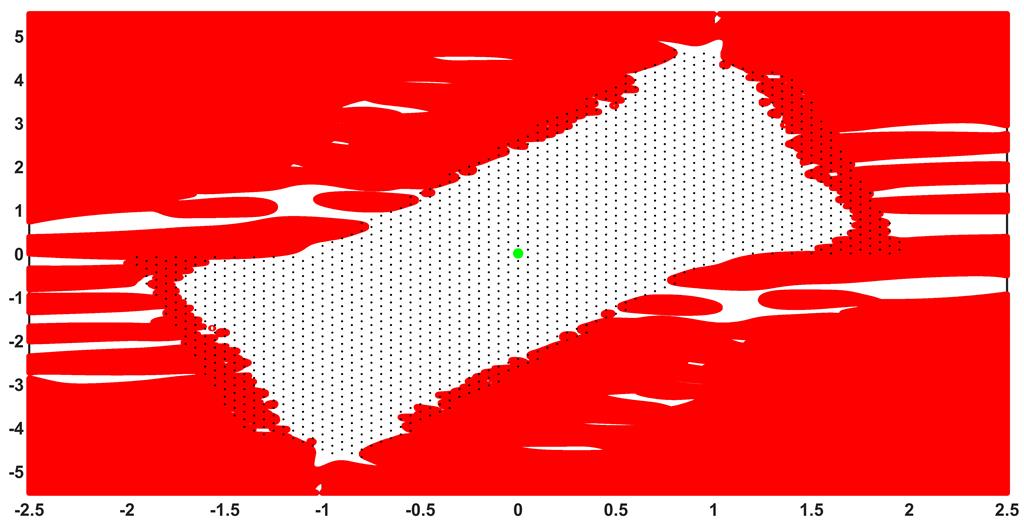}
	\caption{{\footnotesize The black dots show the $1926$ collocation points. The blue stars (left) indicate the vertices where $P(x)$ is not positive definite. The red circles (right) indicate the vertices where \textit{Constraints 4} of the Verification Problem \ref{VP} are not satisfied. The green circle indicates the equilibrium of the system at $(0,0)$, and the triangulation is over the area $[-2.5,~2.5] \times [-5.5,~5.5] $ with $2200^2 $ vertices.}}\label{fig1_pol}
\end{figure}

This example was already used in \cite{giesl2018kernel} and \cite{giesl-cont-1} to illustrate the RBF approximation of the contraction metric and one can compare this result with them.
Here we are  able to rigorously verify the conditions of a contraction metric for the CPA interpolation, while in  previous work it has been checked
for the optimal recovery at finitely many points.

\begin{figure}[ht]	
	\centering
	\includegraphics[height = 5.5cm, width=0.75 \textwidth]{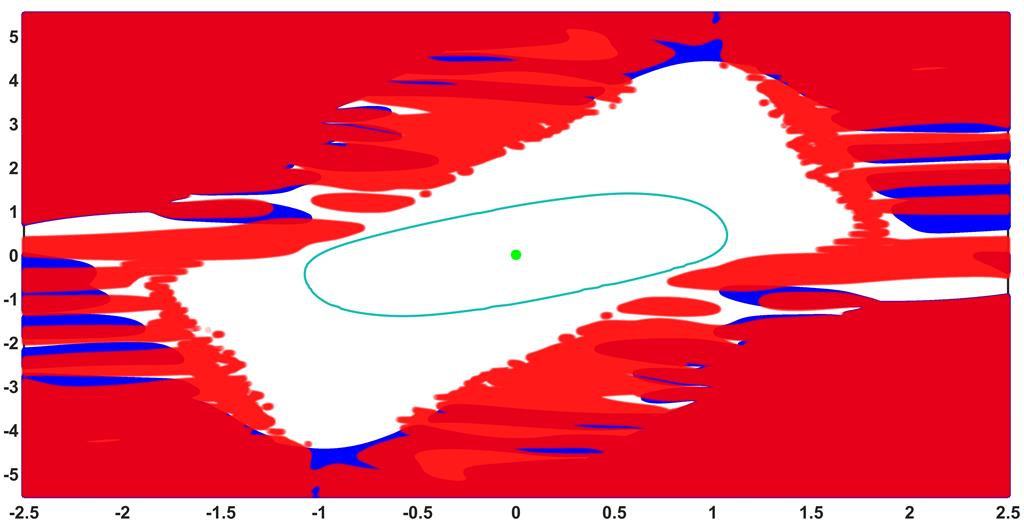}
	\caption{{\footnotesize The suitable area suggested by the method for the contraction metric is in white. The curve inside the white area indicates a positively invariant set, which is a sublevel-set of a computed Lyapunov-like function.
}}\label{fig2_pol}
\end{figure}

 To apply Theorem \ref{th:1} to establish the existence of a unique equilibrium, that then is  necessarily exponentially stable, we additionally need a positively invariant set within the area where the conditions of the
 contraction metric are fulfilled.  To compute such an area we used an approach similar to \cite{GiHa2015combi} and computed a numerical solution to $\nabla V(x)\cdot f(x)=-\sqrt{\delta^2+\|f(x)\|_2^2}$, $x\in\R^2$, using the RBF method, and with $f$ from \eqref{pol} and $\delta=10^{-4}$. Note that an approximate solution will not have negative orbital derivative near the equilibrium, since at the equilibrium $f(x)=0$, see \cite{rbf2007giesl}, so is not a Lyapunov function. However, if the approximation is sufficiently good, then it will have negative orbital derivative outside a neighborhood of the equilibrium. We thus can use
  CPA verification to assert that its orbital derivative is truly negative and then use this information together with level-sets of the computed function $V$ to determine a positively invariant set within the area where
 the metric $P$ is a contraction metric.  We used the same collocation points as above, a kernel given by the Wendland's function $ \psi_{5,3}(cr)=(1-cr)_+^{8} \big(32 (c r)^3+25 (cr)^2 + 8cr+1 \big)$, and $c=0.5$.
 We then used a subsequent CPA verification on a regular $500\times 500$ grid on $[-2.5,2.5]^2$.\\
In Figure \ref{fig2_pol}, we have drawn the largest level set of the computed Lyapunov-like function $V$  that fulfills two conditions:
it is inside of  the area where $P$ is a contraction metric and the level set is in the area where $V$ has negative orbital derivative.
Hence, this sublevel-set is necessarily positively invariant; for more information see \cite[Section 10.XV]{wwode1998}.

\subsection{Speed Control}
As the second example let us consider the system
\begin{eqnarray}
\left\{\begin{array}{lcl}	\dot{x}&=&y\\
	\dot{y}&=& -K_d \, y -x - g x^2 \left( \dfrac{y}{K_d} + x + 1 \right)\end{array}\right.\label{speed control}
\end{eqnarray}
with $ K_d = 1 $ and $g = 6$. The system has two asymptotically stable equilibria $x_0 = (0, 0) $ and $(-0.7887, 0)$,
and the saddle $(-0.2113, 0)$.

\begin{figure}[ht]	
\centering
	\includegraphics[height= 5 cm, width=0.50 \textwidth]{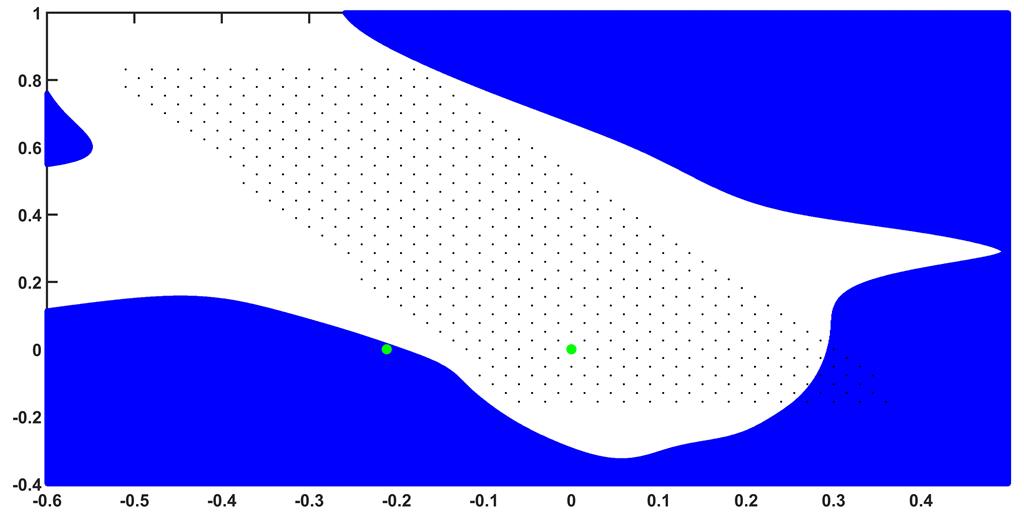}%
	\includegraphics[height= 5 cm, width=0.50 \textwidth]{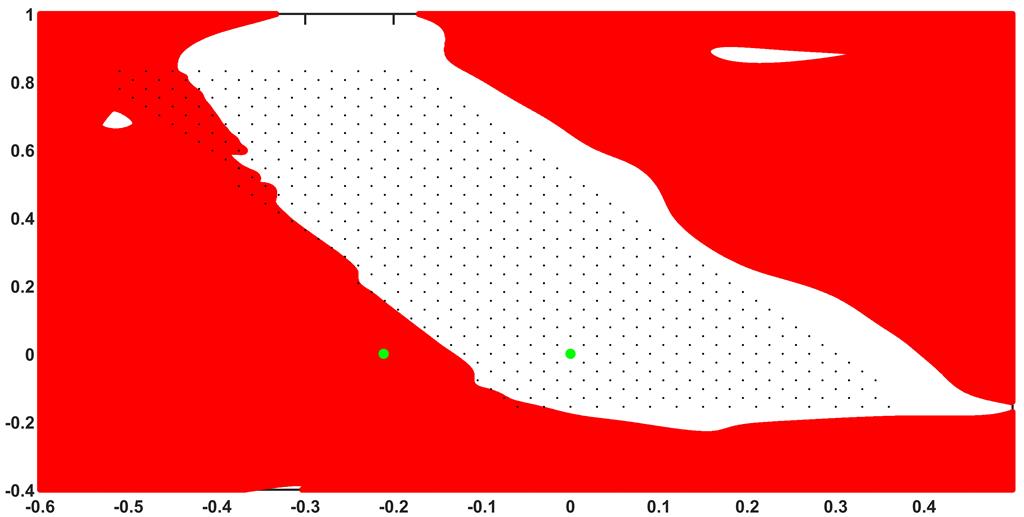}
	\caption{{\footnotesize The black dots show collocation points. The blue stars indicate the vertices where $P(x)$ is not positive definite (left). The red circles indicate the vertices where \textit{Constraints 4} of the Verification Problem \ref{VP} are not satisfied (right), and the green circles are the equilibrium points of the system.}}\label{fig1_speed control}
\end{figure}

The system fails to reach the demanded speed which corresponds to the equilibrium $(0, 0)$ for some inputs since the basin of attraction of $x_0 = (0, 0)$ is not the whole phase space, see \cite[Section 6.1]{rbf2007giesl} for more details.
We provided two sets of collocation points for two equilibria; firstly, $ N = 547 $ points as a hexagonal grid with
$$ X=0.030 \cdot \Z^2 \cap \{(x,y)\in\R^2\,:\, -0.18 \leq y \leq 0.85,  -2.11x-0.3 \leq y \leq -1.79\,x+0.54\}, $$
again with Wendland's function $ \psi_{6,4}(cr)$ and  $c=0.9$.  The triangulation was created over the area $ [-0.6,~0.5] \times [-0.4,~1] $ with $1400^2$ vertices, see Figure \ref{fig1_speed control}.

As in the previous example we computed a solution to $\nabla V(x)\cdot f(x)=-\sqrt{\delta^2+\|f(x)\|^2_2}$ using the RBF method with a subsequent CPA verification.  The procedure and the parameters were the same, the only difference being that we triangulated $[-0.4,0.5]\times[-0.2,0.8]$ for the CPA interpolation.
The results are shown in Figure \ref{fig2_speed control}. The  set inside the white area is a positively invariant set and therefore contains exactly one exponentially stable equilibrium.

\begin{figure}[ht]	
	\centering
	\includegraphics[height = 4.8cm, width=0.80 \textwidth]{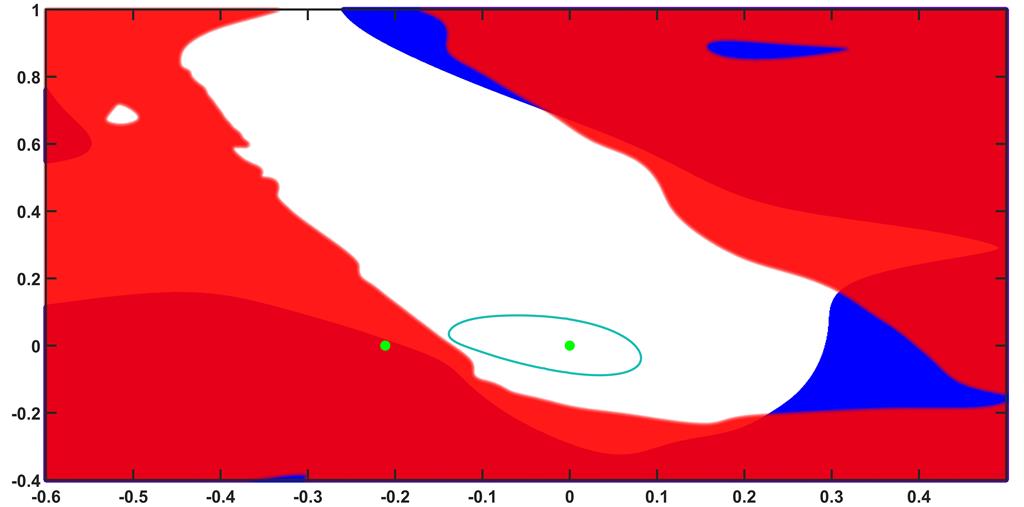}
	\caption{{\footnotesize The suitable area suggested by the method for the contraction metric is in white. The curve inside the white area indicates a positively invariant set, which is a sublevel-set of a computed Lyapunov-like function around the equilibrium at $ (0,0) $.}}\label{fig2_speed control}
\end{figure}

Secondly, $ N = 667 $ collocation points are used for the hexagonal grid around the $(-0.7887, 0)$ equilibrium point, together with triangulation of the area  $ [-1.4,~0] \times [-0.4,~0.4] $ with $1400^2$ vertices, see Figure \ref{fig3_speed control}.
At this equilibrium point, the level sets of a Lyapunov-like function used to estimate the basin of attraction are expanded along the suggested area by our method. Hence, we get a larger positively invariant set.  The Lyapunov function was computed as described above, now with the triangulation for the CPA interpolation on $[-1.4,-0.2]\times[-0.4,0.4]$.

\begin{figure}[ht]	
	\centering
	\includegraphics[height = 4cm,width=0.32 \textwidth]{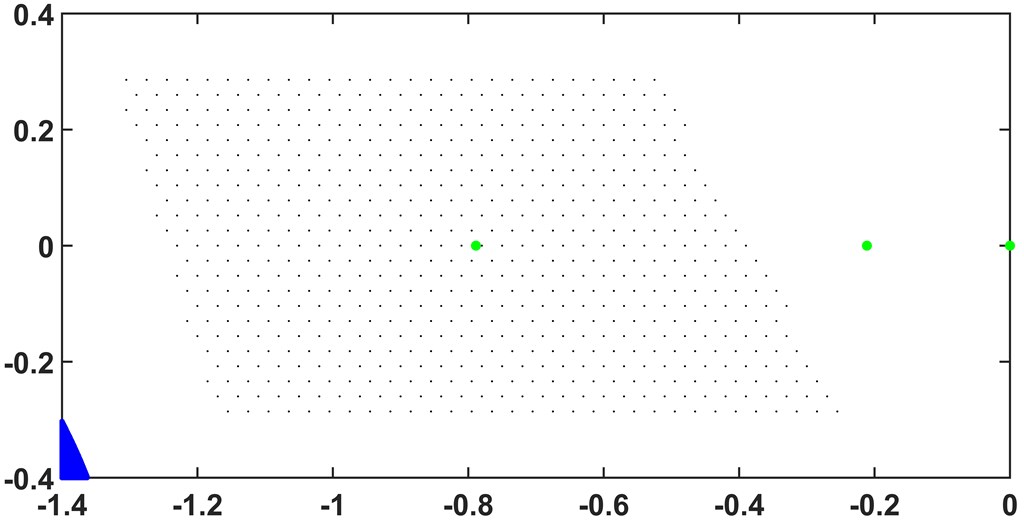}
	\includegraphics[height = 4cm,width=0.32 \textwidth]{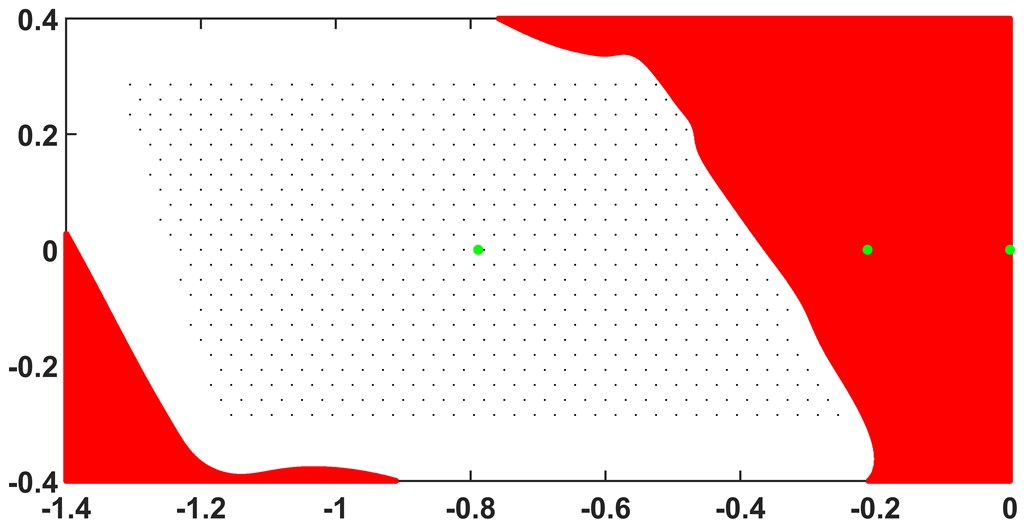}
	\includegraphics[height = 4cm,width=0.32 \textwidth]{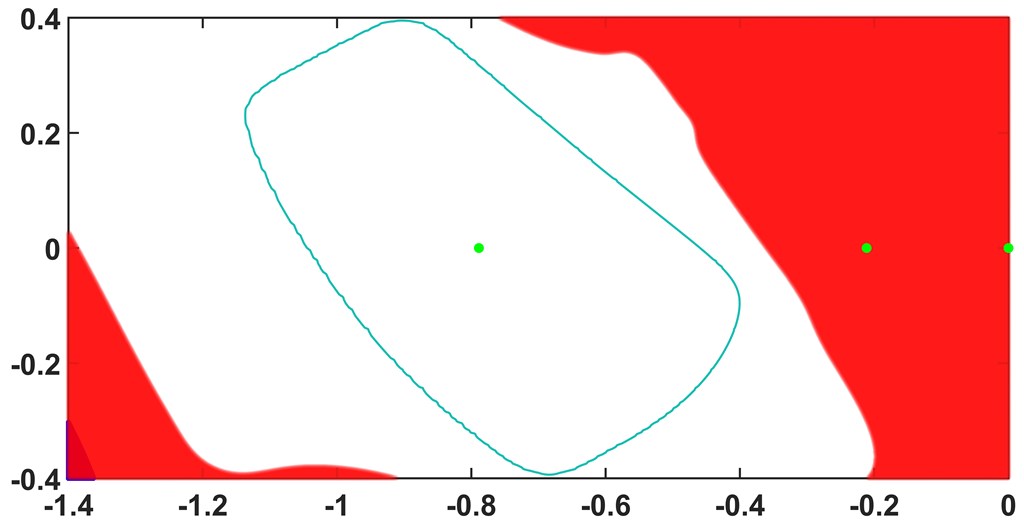}
	\caption{{\footnotesize The black dots show collocation points,
 the blue stars indicate the vertices where $P(x)$ is not positive definite (left), and	the red circles indicate the vertices where \textit{Constraints 4} of the Verification Problem \ref{VP} are not satisfied  (middle). The curve inside the white area indicates the boundary of a positively invariant set (right). The green circles are the equilibrium points of the system.}}\label{fig3_speed control}
\end{figure}

To illustrate the advantage of a contraction metric to, e.g., a Lyapunov function, we now consider the perturbed speed control system with the same parameters $K_d=1$ and $g=6$ as before
\begin{eqnarray}
\left\{\begin{array}{lcl}	\dot{x}&=&y + \epsilon \\
	\dot{y}&=& -y -x -6 (x^2 + \epsilon) \big( y + x + 1 \big),\end{array}\right.\label{speed control-pert}
\end{eqnarray}
first with a small perturbation $ \epsilon = 0.01 $, and then with a large perturbation $ \epsilon = 0.1$. The new system has
three equilibria at $(-0.1359,-0.01)$, $( -0.0780,-0.01)$ and $(-0.776,-0.01)$ for $\epsilon=0.01$ and
only one equilibrium point at $(-0.6648, -0.1)$ for $\epsilon=0.1$. The numerical results (see Figures \ref{fig-speed control pert}, \ref{fig-speed control pert2}) show that our method is robust with respect to perturbations.

\begin{figure}[ht]	
\centering
	\includegraphics[height = 4cm,width=0.45 \textwidth]{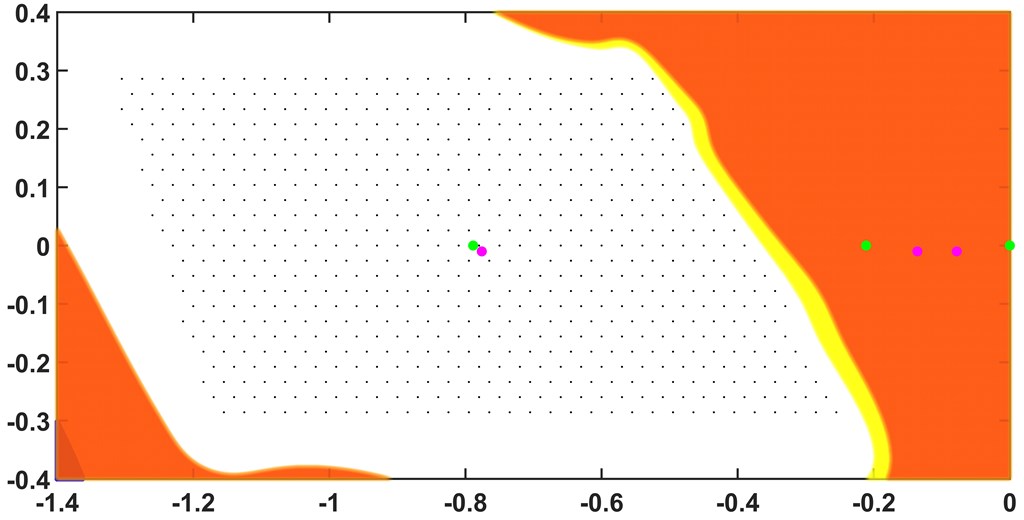}
	\includegraphics[height = 4cm,width=0.45 \textwidth]{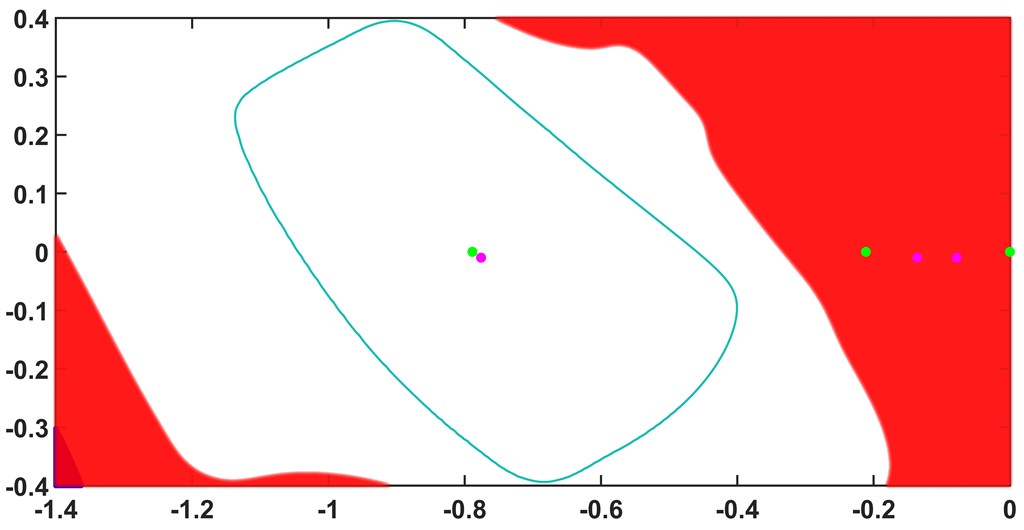}
	\caption{{\footnotesize Small perturbation. The black dots show collocation points, and the blue stars indicate the vertices where $P(x)$ is not positive definite. The vertices where \textit{Constraints 4} of the Verification Problem \ref{VP} are not satisfied are colored in yellow for the original system and in red for the perturbed one (left). The maximum level set of a Lyapunov function inside the suitable area is given by elliptic-like curve (right). The equilibrium points of the original and perturbed system ($\epsilon = 0.01$) are represented by green and magenta circles, respectively.
Note that the results are almost identical to the unperturbed system in Figure \ref{fig3_speed control}, although the equilibria are displaced.
	 }}\label{fig-speed control pert}
\end{figure}

For both the small and the large perturbation we use the same contraction metric  as in the unperturbed system. We can see in left-side plots of Figures \ref{fig-speed control pert}, \ref{fig-speed control pert2} that the metric satisfies the constraints in a very similar area as before for both $ \epsilon = 0.01 $ and $ \epsilon = 0.1 $. However, while the same Lyapunov-like function can be used to determine a positively invariant set for the perturbed system when $ \epsilon = 0.01 $ (see Figure \ref{fig-speed control pert}), we needed to calculate a new Lyapunov-like function for $\epsilon = 0.1$.

\begin{figure}[ht]	
\centering
	\includegraphics[height = 4cm,width=0.45 \textwidth]{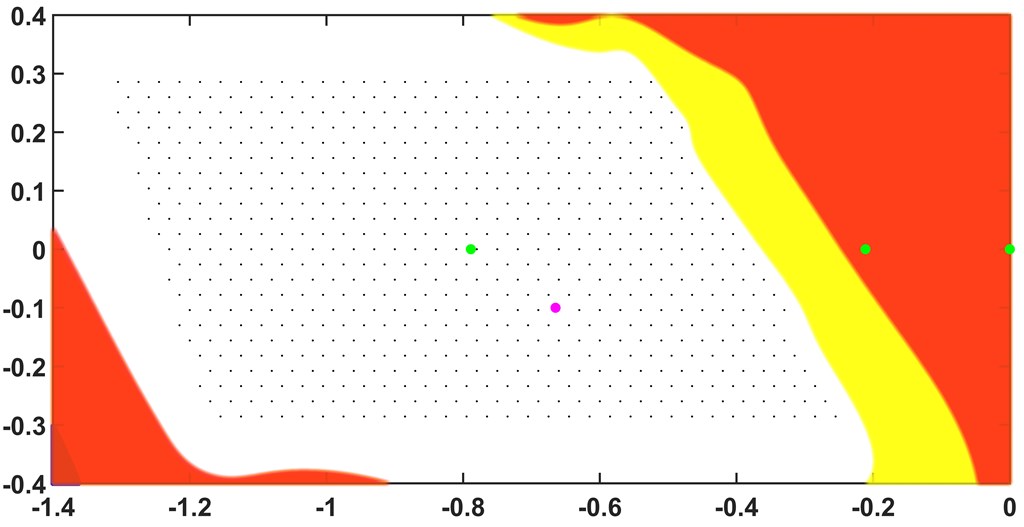}
	\includegraphics[height = 4cm,width=0.45 \textwidth]{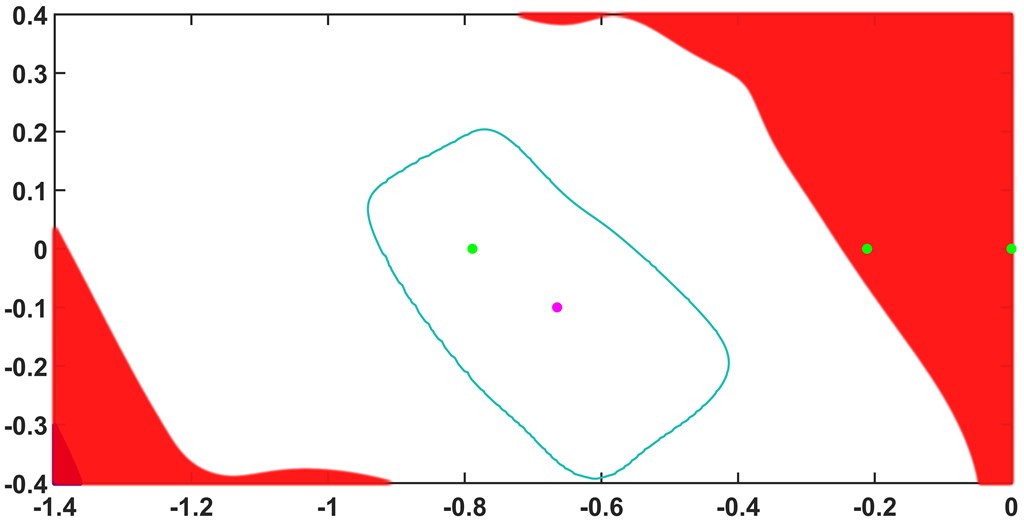}
	\caption{{\footnotesize Large perturbation.  
The black dots show collocation points, and the blue stars (left-down in the background of plots) indicate the vertices where $P(x)$ is not positive definite. The vertices where \textit{Constraints 4} of the Verification Problem \ref{VP} are not satisfied are colored in yellow for the original system and in red for the perturbed one (left). The maximum level set of a Lyapunov function inside the suitable area is given by elliptic-like curve (right). The equilibrium points of the original and perturbed system ($\epsilon = 0.1$) are represented by green and magenta circles, respectively.
Note that the contraction metric for the unperturbed system in Figure \ref{fig3_speed control} is still valid for this system, although the equilibrium has moved a considerable distance.
	 }}\label{fig-speed control pert2}
\end{figure}

\section{Conclusion}

In this paper we have combined two methods to construct and verify a contraction metric for an equilibrium. A contraction metric is a tool to show the stability of an equilibrium and to determine a subset of its basin of attraction. The advantage is that it is robust with respect to perturbations of the dynamical system, including perturbing the position of the equilibrium.

We have combined the RBF method, which is fast and constructs a contraction metric by approximately solving a matrix-valued PDE with meshfree collocation, with the CPA method, which interpolates the RBF metric by a continuous function, which is affine on each simplex of a fixed triangulation. The CPA method includes a rigorous verification that the computed metric is in fact a contraction metric. The new combined method is as fast as the RBF method, but also includes a rigorous verification, which was missing in the original RBF method. We have also shown in the paper that this combined method always succeeds in rigorously constructing a contraction metric by making the set of collocation points and the triangulation finer and finer.

When compared to other methods to determine the basin of attraction of an equilibrium, e.g. Lyapunov functions, the computation of a contraction metric is computationally more demanding as we construct a matrix-valued function, but it is robust with respect to perturbations of the system. One can combine these two approaches by first computing a   Lyapunov function, which will have a strictly negative orbital derivative in some areas, but will exhibit some areas, where the orbital derivative is non-negative. If a sub-level set of the Lyapunov function covers this area, then we have found a positively invariant set and can then apply the method of this paper to prove that there is a unique equilibrium in this sub-level set, and the sub-level set is part of its basin of attraction.

%
%

\label{References}
\begin{small}
\bibliographystyle{amsplain-nodash}
\bibliography{mainbibadapted} 

\providecommand{\bysame}{\leavevmode\hbox to3em{\hrulefill}\thinspace}
\providecommand{\MR}{\relax\ifhmode\unskip\space\fi MR }
\providecommand{\MRhref}[2]{%
  \href{http://www.ams.org/mathscinet-getitem?mr=#1}{#2}
}
\providecommand{\href}[2]{#2}
\begin{thebibliography}{10}

\bibitem{Aghannan2003}
N.~Aghannan and P.~Rouchon., \emph{An intrinsic observer for a class of
  {L}agrangian systems}, IEEE Trans. Automat. Control \textbf{48} (2003),
  936--944.

\bibitem{AnPaXXX}
J.~Anderson and A.~Papachristodoulou, \emph{Advances in computational
  {L}yapunov analysis using sum-of-squares programming}, Discrete Contin. Dyn.
  Syst. Ser. B \textbf{20} (2015), no.~8, 2361--2381.

\bibitem{Aylward2008}
E.~Aylward, P.~Parrilo, and J.-J. Slotine, \emph{Stability and robustness
  analysis of nonlinear systems via contraction metrics and {SOS} programming},
  Automatica \textbf{44} (2008), 2163--2170.

\bibitem{baier2012linear}
R.~Baier, L.~Gr{\"u}ne, and S.~Hafstein, \emph{Linear programming based
  {L}yapunov function computation for differential inclusions}, Discrete
  Contin. Dyn. Syst. Ser. B \textbf{17} (2012), no.~1, 33--56.

\bibitem{BGHKLXXXXmultiattr}
J.~Bj\"ornsson, P.~Giesl, S.~Hafstein, C.~Kellett, and H.~Li, \emph{Computation
  of {L}yapunov functions for systems with multiple attractors}, Discrete
  Contin. Dyn. Syst. Ser. A \textbf{35} (2015), no.~9, 4019--4039.

\bibitem{chesibook}
G.~Chesi, \emph{Domain of {A}ttraction: {A}nalysis and {C}ontrol via {SOS}
  {P}rogramming}, Springer, 2011.

\bibitem{rbf2007giesl}
P.~Giesl, \emph{Construction of global {L}yapunov functions using radial basis
  functions}, Lecture Notes in Mathematics, vol. 1904, Springer-Verlag, Berlin,
  2007.

\bibitem{contr2015giesl}
P.~Giesl, \emph{Converse theorems on contraction metrics for an equilibrium},
  J. Math. Anal. Appl. (2015), no.~424, 1380--1403.

\bibitem{GiHa2013CPAmetric}
P.~Giesl and S.~Hafstein, \emph{Construction of a {CPA} contraction metric for
  periodic orbits using semidefinite optimization}, Nonlinear Anal. \textbf{86}
  (2013), 114--134.

\bibitem{Giesl2014}
P.~Giesl and S.~Hafstein, \emph{Revised {CPA} method to compute {L}yapunov
  functions for nonlinear systems}, J. Math. Anal. Appl. \textbf{410} (2014),
  292--306.

\bibitem{GiHa2015combi}
P.~Giesl and S.~Hafstein, \emph{Computation and verification of {L}yapunov
  functions}, SIAM Journal on Applied Dynamical Systems \textbf{14} (2015),
  no.~4, 1663--1698.

\bibitem{giesl2018kernel}
P.~Giesl and H.~Wendland, \emph{Kernel-based discretization for solving
  matrix-valued {PDE}s}, SIAM J. Numer. Anal. \textbf{56} (2018), no.~6,
  3386--3406.

\bibitem{giesl-cont-1}
P.~Giesl and H.~Wendland, \emph{Construction of a contraction metric by
  meshless collocation}, Discrete Contin. Dyn. Syst. Ser. B \textbf{24} (2019),
  no.~8, 3843--3863.

\bibitem{Haf2004exex}
S.~Hafstein, \emph{A constructive converse {L}yapunov theorem on exponential
  stability}, Discrete Contin. Dyn. Syst. \textbf{10} (2004), no.~3, 657--678.

\bibitem{HaKa2019datalimit}
S.~Hafstein and C.~Kawan, \emph{Numerical approximation of the data-rate limit
  for state estimation under communication constraints}, J. Math. Anal. Appl.
  \textbf{473} (2019), no.~2, 1280--1304.

\bibitem{hafstein2017study}
S~Hafstein and A~Valfells, \emph{Study of dynamical systems by fast numerical
  computation of {L}yapunov functions}, Proceedings of the 14th International
  Conference on Dynamical Systems: Theory and Applications (DSTA), Mathematical
  and Numerical Aspects of Dynamical System Analysis, vol. Mathematical and
  Numerical Aspects of Dynamical System Analysis Lodz, Poland, 2017,
  pp.~229--240.

\bibitem{HaVa2019intLya}
S.~Hafstein and A.~Valfells, \emph{Efficient computation of {L}yapunov
  functions for nonlinear systems by integrating numerical solutions},
  Nonlinear Dynamics (To be published 2019).

\bibitem{hahn}
W.~Hahn, \emph{Stability of motion}, Springer, Berlin, 1967.

\bibitem{iske2}
A.~Iske, \emph{Perfect centre placement for radial basis function methods},
  Tech. Report TUM-M9809, TU Munich, Germany, 1998.

\bibitem{Johansen2000collo}
T.~Johansen, \emph{Computation of {L}yapunov functions for smooth, nonlinear
  systems using convex optimization}, Automatica \textbf{36} (2000),
  1617--1626.

\bibitem{JGD1999cpajulian}
P.~Julian, J.~Guivant, and A.~Desages, \emph{A parametrization of piecewise
  linear {L}yapunov functions via linear programming}, Int. J. Control
  \textbf{72} (1999), no.~7-8, 702--715.

\bibitem{KaPeXXXXalttosos}
R.~Kamyar and M.~Peet, \emph{Polynomial optimization with applications to
  stability analysis and control -- an alternative to sum of squares}, Discrete
  Contin. Dyn. Syst. Ser. B \textbf{20} (2015), no.~8, 2383--2417.

\bibitem{Khalil2002}
H.~Khalil, \emph{Nonlinear systems}, 3. ed., Pearson, 2002.

\bibitem{krasovskii}
N.~N. Krasovskii, \emph{Problems of the theory of stability of motion}, Mir,
  Moskow, 1959, English translation by Stanford University Press, 1963.

\bibitem{lohmiller}
W.~Lohmiller and J.-J. Slotine, \emph{On contraction analysis for non-linear
  systems}, Automatica \textbf{34} (1998), 683--696.

\bibitem{lyapunovbook1907endtrans}
A.~M. Lyapunov, \emph{The general problem of the stability of motion},
  Internat. J. Control \textbf{55} (1992), no.~3, 521--790, Translated by A. T.
  Fuller from {\'E}douard Davaux's French translation (1907) of the 1892
  Russian original, With an editorial (historical introduction) by Fuller, a
  biography of Lyapunov by V. I. Smirnov, and the bibliography of Lyapunov's
  works collected by J. F. Barrett, Lyapunov centenary issue.

\bibitem{Mar2002cpa}
S.~Marin\'osson, \emph{Lyapunov function construction for ordinary differential
  equations with linear programming}, Dynamical Systems: An International
  Journal \textbf{17} (2002), 137--150.

\bibitem{Papachristodoulou2013}
A.~Papachristodoulou, J.~Anderson, G.~Valmorbida, S.~Pranja, P.~Seiler, and
  P.~Parrilo, \emph{{SOSTOOLS}: Sum of squares optimization toolbox for
  {MATLAB}}, version 3.00 ed., User's guide, 2013.

\bibitem{sosphdthesis2000parrilo}
P.~Parrilo, \emph{Structured semidefinite programs and semialgebraic geometry
  methods in robustness and optimiziation}, PhD thesis: California Institute of
  Technology Pasadena, California, 2000.

\bibitem{RaSh2010branchrelaxlya}
S.~Ratschan and Z.~She, \emph{Providing a basin of attraction to a target
  region of polynomial systems by computation of {L}yapunov-like functions},
  SIAM J. Control Optim. \textbf{48} (2010), no.~7, 4377--4394.

\bibitem{Vannelli1985}
A.~Vannelli and M.~Vidyasagar, \emph{Maximal {L}yapunov functions and domains
  of attraction for autonomous nonlinear systems}, Automatica \textbf{21}
  (1985), no.~1, 69--80.

\bibitem{vidya2002}
M.~Vidyasagar, \emph{Nonlinear system analysis}, 2. ed., Classics in applied
  mathematics, SIAM, 2002.

\bibitem{wwode1998}
W.~Walter, \emph{Ordinary differential equation}, Springer, 1998.

\bibitem{Wendland1998}
H.~Wendland, \emph{Error estimates for interpolation by compactly supported
  {R}adial {B}asis {F}unctions of minimal degree}, J. Approx. Theory
  \textbf{93} (1998), 258--272.

\bibitem{wendland2005scattered}
H.~Wendland, \emph{Scattered data approximation}, vol.~17, Cambridge university
  press, 2005.

\bibitem{zubovbook1964}
V.~I. Zubov, \emph{Methods of {A}. {M}. {L}yapunov and their application},
  Translation prepared under the auspices of the United States Atomic Energy
  Commission; edited by Leo F. Boron, P. Noordhoff Ltd, Groningen, 1964.

\end{thebibliography}
\end{small}
\addcontentsline{toc}{chapter}{References}

\end{document}